\documentclass[reqno,12pt]{amsart}        %For Frank in the US
\usepackage{amsfonts,amsmath,amssymb,amsthm,colordvi,mathrsfs,graphicx}
\usepackage{caption,subcaption,float}
\usepackage[utf8]{inputenc}
\usepackage{mathrsfs}
\usepackage{geometry}
\usepackage{enumerate}

\usepackage{amsmath, amssymb, amsthm, geometry, hyperref}

\usepackage{tabularx}
\usepackage{tabulary}
 
\hypersetup{
  colorlinks=true,
  linkcolor=blue,
  citecolor=blue,
  urlcolor=blue
}

\usepackage[all,knot,poly]{xy}
\usepackage{tikz-cd}
\usepackage{tikz}
\usepackage{verbatim}
\usetikzlibrary{arrows}
\usetikzlibrary{knots}
\tikzset{every path/.style={line width=0.4pt},every node/.style={transform shape,knot crossing,inner sep=1.5pt},>=triangle 60,text node/.style={rectangle,transform shape=false,black}}
 \usetikzlibrary{decorations.markings, arrows.meta}

\theoremstyle{plain}      
\newtheorem{thm}{Theorem}[section]     
\newtheorem{theorem}[thm]{Theorem}     
\newtheorem{corollary}[thm]{Corollary}     
\newtheorem{lemma}[thm]{Lemma}

\theoremstyle{remark}

\newtheorem{remark}[thm]{Remark} 
     
\theoremstyle{definition}      
\newtheorem{definition}[thm]{Definition}

\newcommand{\OO}{\mathcal{O}}
\newcommand{\Gm}{\mathbb{G}_m}
\newcommand{\Spec}{\operatorname{Spec}}
\newcommand{\Proj}{\operatorname{Proj}}
\newcommand{\Hom}{\operatorname{Hom}}
\newcommand{\Ext}{\operatorname{Ext}}
\newcommand{\cL}{\mathcal{L}}
\newcommand{\kfield}{\mathbb{C}}
 \newcommand{\Sym}{\operatorname{Sym}}
 \newcommand{\PP}{\mathbb{P}}
 \newcommand{\Der}{\operatorname{Der}}
\newcommand{\LL}{\mathcal{L}}

\newcommand{\m}{\mathfrak{m}}
\newcommand{\cO}{\mathcal{O}}

 \newcommand{\kk}{\mathbf{k}}

\newcommand{\C}{{\mathbb C}}

\newcommand{\T}{{\mathbb T}}

\newcommand{\Z}{{\mathbb Z}}

\newcommand{\depth}{\operatorname{depth}}

\title[]{}
\begin{document}

\title{On infinitesimal deformations of singular curves} %    On the non-rigidity of singular curves} infinitesimal deformations

\author{Mounir Nisse}
\address{Mounir Nisse\\
Department of Mathematics, Xiamen University Malaysia, Jalan Sunsuria, Bandar Sunsuria, 43900, Sepang, Selangor, Malaysia.
}
\email{mounir.nisse@gmail.com, mounir.nisse@xmu.edu.my}
\thanks{}

\title{On infinitesimal deformations of singular varieties II}
%{Combinatorial Coamoebas Structure, and Applications}

\maketitle

\begin{abstract}
The deformation theory of affine cones over polarized projective varieties, initiated by Pinkham and further developed by Schlessinger and Wahl, is central to the study of singularities and graded deformation functors.  
For a projective variety \(Y\) with ample line bundle \(\mathcal L\), the affine cone \(C(Y)\) carries a natural \(\mathbb Z\)-grading, and Pinkham's classical result identifies the graded pieces of its first-order deformation space:
\[
    T^{1}(C(Y))_m \,\cong\, 
    H^{1}\!\left(Y,\,T_Y \otimes \mathcal L^{\otimes m}\right).
\]
This expresses that deformations of \(C(Y)\) come from weighted deformations of \((Y,\mathcal L)\), with negative weights corresponding to smoothings and nonnegative weights to embedded deformations.
We give a streamlined proof of this isomorphism for possibly singular \(Y\), using reflexive differentials and \(\mathbb G_m\)-equivariant deformations of the punctured cone.  We then compute graded deformation spaces for several examples, illustrating phenomena from smoothability to rigidity.  As an application, we obtain a practical rigidity criterion: if
\(
   H^1(Y,\,T_Y \otimes \mathcal L^{\otimes m})=0
   \text{ for all } m\in\mathbb Z,
\)
then \(C(Y)\) is rigid.  We exhibit explicit polarized varieties satisfying these vanishings, producing new rigid affine cones.
\end{abstract}

%%%%%%%%%%%%%%%%%%%%%%%%%%%%%%%%%%%%%%%%%%%%%%%%%%%%%%%%%%%%%%%%%%%%%%%%%%% 
%%%%%%%%%%%%%%%%%%%%%%%%%%%%%%%%%%%%%%%%%%%%%%%%%%%%%%%%%%%%%%%%%%%%%%%%%%%

\section{introduction}

 Let $Y$ be a projective variety over an algebraically closed field
of characteristic zero, and let $\mathcal L$ be an ample line bundle on
$Y$.  The associated affine cone
\[
   C(Y)
   := \operatorname{Spec}\!\left(
         \bigoplus_{m\ge0}
            H^0(Y,\mathcal L^{\otimes m})
      \right)
\]
is a normal variety whose vertex usually carries a nontrivial isolated
singularity, even when $Y$ is smooth.  When $Y$ has mild singularities
(e.g. Kawamata canonical, or log terminal), the cone remains normal and
$\mathbb Q$-Gorenstein, and its deformation theory behaves well.
The infinitesimal deformation theory of $C(Y)$ is described by the graded
module
\[
   T^1(C(Y))
   := \operatorname{Ext}^1(\Omega_{C(Y)}, \mathcal O_{C(Y)}),
\qquad
   T^1(C(Y))=\bigoplus_{m\in\mathbb Z}T^1(C(Y))_m,
\]
reflecting the $\mathbb G_m$-action on the cone.  A classical theorem
due to Pinkham \cite{P-74}, refined by Schlessinger--Stasheff
\cite{SS-79}, Wahl \cite{W-77}, and others, identifies
these graded pieces with intrinsic cohomology groups on $Y$:
\begin{equation}\label{eq:graded-identification}
   T^1(C(Y))_m
   \;\cong\;
   H^1\!\left(Y,\, T_Y \otimes \mathcal L^{\otimes m}\right),
\end{equation}
where $T_Y:=\mathcal{H}om(\Omega_Y,\mathcal O_Y)$ is the reflexive
tangent sheaf.  The isomorphism   $(1)$ % 
is obtained by studying $\mathbb G_m$-equivariant deformations of the
punctured cone $C(Y)\setminus\{0\}$ and descending them to weighted
deformations of the pair $(Y,\mathcal L)$.

The graded decomposition of $T^1(C(Y))$ admits a clear geometric
interpretation: negative degrees $m<0$ correspond to possible directions
that smooth the vertex singularity, while the nonnegative components
$m\ge0$ encode embedded, polarized, or locally trivial deformations
originating from the projective embedding $Y\hookrightarrow\mathbb P$.
This translation of deformation-theoretic questions on a singular affine
variety into computable cohomological questions on a (possibly singular)
projective variety makes the theory both powerful and accessible.

\subsection{Local versus global deformation theory}

A conceptual understanding of deformations of affine cones requires
distinguishing between \emph{local} and \emph{global} deformation
theory.  The local deformation functor of the singularity $(C(Y),0)$ is
controlled by
\[
   T^1_{\mathrm{loc}}(C(Y),0)
     = \Ext^1_{\mathcal O_{C(Y),0}}
         (\Omega_{C(Y),0},\mathcal O_{C(Y),0}),
\]
which depends only on the analytic germ.  In contrast, global
deformations of $C(Y)$ are governed by the global sheaf $\mathcal T^1$
and the module $T^1(C(Y))$.
For cones, the relationship between local and global deformation theory
is mediated by the geometry of the punctured cone.  Pinkham showed that
$\mathbb G_m$-equivariant deformations of $C(Y)\setminus\{0\}$
correspond to weighted deformations of $(Y,\mathcal L)$, and every local
deformation of the vertex arises from a negative weight piece.  Thus,
\[
   T^1_{\mathrm{loc}}(C(Y),0)
   \;\cong\;
     \bigoplus_{m<0} T^1(C(Y))_m,
\qquad
   T^1_{\mathrm{glob}}(C(Y))
   \;\cong\;
     \bigoplus_{m\in\mathbb Z} T^1(C(Y))_m,
\]
providing a splitting of the deformation theory into smoothing
directions (negative $m$) and embedded/polarized directions (nonnegative
$m$).

\subsection{Reflexive differentials for singular $Y$}

When $Y$ is singular, the cotangent sheaf $\Omega_Y$ need not be locally
free, but it remains torsion-free.  Its reflexive hull
\(
   \Omega_Y^{[1]} := (\Omega_Y)^{\vee\vee}
\)
is coherent and locally free in codimension one.  The tangent sheaf is
defined as its dual
\(
   T_Y := (\Omega_Y^{[1]})^\vee,
\)
and governs first-order deformations via the usual rule
\[
   \mathrm{Def}_Y(k[\varepsilon]/\varepsilon^2)
   \cong
   H^1(Y, T_Y).
\]
When $Y$ has only  Kawamata  log terminal or canonical singularities, many standard results
on reflexive differentials hold, as extension theorems, cohomology
vanishing, and compatibility with pullback under crepant morphisms.  In
this setting, \eqref{eq:graded-identification} continues to hold using
reflexive differentials, making the deformation theory of cones over
singular varieties nearly as tractable as in the smooth case.

\subsection{Obstruction theory and higher $\mathbf{T^i}$}

The higher Ext-group
\[
   T^2(C(Y)) := \Ext^2(\Omega_{C(Y)},\mathcal O_{C(Y)})
\]
governs obstructions to lifting infinitesimal deformations.  As for
$T^1(C(Y))$, it inherits a natural $\mathbb Z$-grading:
\[
   T^2(C(Y))=\bigoplus_{m\in\mathbb Z}T^2(C(Y))_m.
\]
Under suitable conditions (e.g.\ $Y$ smooth, or $Y$ klt and $\mathcal L$
ample), one obtains identifications
\[
   T^2(C(Y))_m
   \;\cong\;
   H^2\!\left(Y, T_Y\otimes\mathcal L^{\otimes m}\right).
\]
Negative-weight obstructions detect whether smoothing directions lift to
higher order; positive-weight obstructions control extensions of
embedded deformations.

The vanishing of $T^2(C(Y))$ ensures unobstructedness of the entire
deformation functor.  More modest vanishing statements such as
\[
   H^2(Y,T_Y\otimes\mathcal L^m)=0
   \quad \text{for all } m<0,
\]
show that all smoothing directions are unobstructed.  Such criteria are
widely used in the study of smoothability of Fano cone singularities and
Calabi--Yau cones.

Moreover, with Yen-Kheng Lim, we have a continuation of this work with some 
applications in physics \cite{LN1-25},and \cite{LN2-25}. Indeed, In 
gravitational physics, deformation theory provides a systematic language
for describing small perturbations of geometric backgrounds, including
black hole spacetimes and their near-horizon limits.
Such deformations encode physical fluctuations, stability properties, and
moduli of solutions, and play a key role in understanding how classical
geometries respond to quantum or stringy corrections.
In holographic settings, deformations of the underlying geometry translate
into deformations of the dual quantum theory, offering insight into horizon
physics and the behavior of observables across event horizons. This is an ongoing work 
in preparation.

\subsection{Examples}

\subsubsection*{(1) Cones over curves}

Let $Y=C$ be a smooth projective curve of genus $g$, with $\deg\mathcal
L=d>0$. Then
\[
   T_C \cong \omega_C^{-1},
\qquad
   H^1(C,T_C\otimes\mathcal L^m)
      \cong H^0(C, \omega_C^{\otimes2}\otimes\mathcal L^{-m})^\vee,
\]
so the graded pieces of $T^1(C(Y))$ can be computed explicitly in terms
of linear series on the curve.  Cones over canonical curves often admit
nontrivial smoothings; cones over highly positive embeddings are
frequently rigid.

\subsubsection*{(2) Cones over del Pezzo surfaces}

If $Y$ is del Pezzo and $\mathcal L=-K_Y$, then $C(Y)$ is a threefold
Fano cone singularity.  Many such cones satisfy strong vanishing
statements for $H^1(Y,T_Y\otimes(-K_Y)^m)$, leading to rigidity
($T^1(C(Y))=0$).  These rigid cones are central objects in the theory of
K-moduli and often admit Kähler--Einstein metrics.

\subsubsection*{(3) Cones over K3 surfaces}

For a polarized K3 surface $(Y,\mathcal L)$, the cone $C(Y)$ is a
Calabi--Yau threefold singularity.  Although $H^1(Y,T_Y)=0$, the twisted
groups $H^1(Y,T_Y\otimes\mathcal L^m)$ may be nonzero, leading to rich
smoothing behavior.  Such cones appear in the study of degenerations of
Calabi--Yau threefolds, mirror symmetry, and compactifications of K3
moduli.

\subsubsection*{(4) Toric cones}

If $Y$ is a projective toric variety and $\mathcal L$ is an ample toric
line bundle, then $C(Y)$ is a normal affine toric variety.  In this
case, $T^i(C(Y))$ can be expressed combinatorially using the Stanley--Reisner
complex or the local cohomology of the semigroup ring.  These formulas
describe the graded pieces explicitly in terms of the fan or polytope
associated to $(Y,\mathcal L)$, and frequently reveal rigidity.

\subsubsection*{(5) Cones over Enriques surfaces}

Let $Y$ be an Enriques surface with a base-point-free ample line bundle
$\mathcal L$.  Then $C(Y)$ is a canonical threefold singularity.
Because $T_Y$ has $H^1(Y,T_Y)$ one-dimensional (coming from deformations
of the Enriques surface), one finds nontrivial $T^1(C(Y))_0$, and the
negative pieces depend delicately on the Fourier--Mukai partners of $Y$.
These cones provide examples of canonical but not terminal singularities
with subtle smoothing theory.

\subsection{Applications to moduli spaces}

The deformation theory of affine cones is deeply intertwined with moduli
theory.  When $(Y,\mathcal L)$ varies in a moduli space of polarized
varieties, the family of cones $C(Y)$ captures first-order variations of
both the underlying variety and its polarization:
\[
   T^1(C(Y))_m
     \;\cong\;
   H^1(Y, T_Y\otimes\mathcal L^m)
   = \text{tangent directions in the moduli stack (weighted)}.
\]
Conversely, cones arise as central fibers of degenerations or test
configurations.  In K-stability theory, if $Y$ is Fano and
$\mathcal L=-K_Y$, then $C(Y)$ is a model singularity appearing in the
study of Kähler--Einstein metrics.  The graded pieces of $T^1(C(Y))$
describe how $Y$ may degenerate within the K-moduli stack, and the
vanishing of negative pieces implies rigidity of the Fano variety under
degenerations.
More generally, the identifications
\[
   T^i(C(Y))_m
      \;\cong\;
   H^i(Y, T_Y\otimes\mathcal L^m)
\]
provide a bridge between the deformation theory of cones and the
tangent--obstruction theory of moduli spaces.  This connection is
especially useful in understanding compactifications, boundary divisors,
and wall-crossing phenomena.

\subsection{Comparison with analytic deformation theory}

Analytic methods provide a parallel framework for studying deformations
of singularities.  For an isolated complex analytic singularity $(X,0)$,
the analytic deformation functor has tangent space
\[
   T^1_{\mathrm{an}}(X,0)
   = \Ext^1_{\mathcal O_{X,0}}(\Omega_{X,0},\mathcal O_{X,0}),
\]
identical to the algebraic version, but its interpretation is enriched
by Hodge theory, mixed Hodge structures, and vanishing cycle techniques.

For cones $C(Y)$ with $Y$ smooth, analytic techniques relate negative
graded pieces of $T^1$ to variations of the Milnor fiber and to
spectrum exponents of the singularity.  When $Y$ is Calabi--Yau, one
obtains mixed Hodge structures closely analogous to those arising in
degenerations of Calabi--Yau threefolds.

Comparison theorems (e.g.\ GAGA and Artin approximation) ensure that for
projective $Y$ the analytic and algebraic deformation theories of
$C(Y)$ coincide.  However, the analytic viewpoint gives additional
insight into smoothing components, monodromy, and vanishing cycles,
particularly useful for cones over K3 surfaces or Calabi--Yau varieties.

\vspace{0.3cm} 

%\newpage

% 
The affine cone \( C(Y)\)
is a fundamental object in birational geometry and deformation
theory.  Its vertex is typically a singular point, and the 
infinitesimal deformations of this singularity are encoded in the
graded module
\[
  T^1(C(Y)) = \operatorname{Ext}^1(\Omega_{C(Y)},\mathcal{O}_{C(Y)}),
\]
whose $m$-th graded summand is denoted $T^1(C(Y))_m$.

A classical result, originating in the work of Pinkham and later
developed by Schlessinger–Stasheff, Wahl, and many others,
identifies these graded pieces with cohomology groups intrinsic
to the base variety $Y$.  Concretely, one has the isomorphism $(1)$
where $T_Y$ denotes the tangent sheaf of $Y$.
The appearance of $\mathcal{L}^{\otimes m}$ reflects the natural
$\mathbb{Z}$-grading on the coordinate ring of the cone, and the
graded deformation functor respects this structure.

The identification  $(1)$ % 
is obtained by
relating deformations of the cone to $\mathbb{G}_m$-equivariant
deformations of the punctured cone and then descending them
to deformations of $Y$.  Intuitively, each graded piece measures
deformations of $Y$ ``weighted'' by the twisting line bundle
$\mathcal{L}^{\otimes m}$.  The negative graded pieces correspond
to deformations that may smooth the vertex of the cone, while the
nonnegative pieces record embedded or polarized deformations
arising from the projective embedding $Y \hookrightarrow \mathbb{P}$.

This description provides a powerful and computable method for
studying the deformation theory of affine cones.  It has important
applications to the rigidity of cones over Fano varieties,
smoothing  problems, unobstructed ness results, and in the
analysis of Kähler–Einstein metrics on singular varieties.

%%%%%%%%%%%%%%%%%%%%%%%%%%%%%%%%%%%%%%%%%%%%%%%%%%%%%%%%%%%%%%%%%%%%%%%%%%% 
%%%%%%%%%%%%%%%%%%%%%%%%%%%%%%%%%%%%%%%%%%%%%%%%%%%%%%%%%%%%%%%%%%%%%%%%%%%

 \section{Preliminaries}% 

In algebraic deformation theory, the algebra
\(
A_1 := \mathbf{C}[\varepsilon]/(\varepsilon^2)
\)
is called the \emph{dual numbers}. It satisfies:
\(
\varepsilon^2 = 0,\quad \varepsilon \neq 0.
\)
The spectrum
\(
\Spec A_1
\)
is a one-point scheme whose structure sheaf contains an infinitesimal nilpotent direction. Geometrically, it represents a ``first--order infinitesimal thickening'' of a point. This base plays the role of the set of \emph{first-order parameters}.

\begin{definition}
Let \(Y\) be a scheme over \(\mathbb{C}\). A \textit{first--order (abstract) deformation} of \(Y\) is a pair
\(
(\mathcal{Y}, \pi),
\)
where:
\begin{itemize}
  \item[(i)] \(\pi : \mathcal{Y} \to \Spec A_1\) is a flat morphism of schemes;
  \item[(ii)] the special fiber is isomorphic to \(Y\), i.e.
  \(
  \mathcal{Y} \times_{\Spec A_1} \Spec \mathbb{C} \;\cong\; Y.
  \)
\end{itemize}
\end{definition}

\noindent  %
 Intuitively, flatness means that the fibers vary in a ``continuous'' algebraic sense.
Two first--order deformations
\(
\pi:\mathcal{Y}\to\Spec A_1,\)% 
\, and  \(\pi':\mathcal{Y}'\to\Spec A_1
\)
are considered \textit{isomorphic} if there exists an isomorphism
\(
\Phi:\mathcal{Y} \longrightarrow \mathcal{Y}'
\)
such that: (i) \(\pi' \circ \Phi = \pi\) (compatibility with the projection), and (ii)  \(\Phi|_Y = \mathrm{id}_Y\) on the special fiber.

\subsection{Local algebra description}

Let \(\OO_{\mathcal{Y}}\) be the structure sheaf of the total space \(\mathcal{Y}\).
Since \(\varepsilon^2=0\), flatness of 
\(\mathcal{Y} \to \Spec A_1\)
implies that as sheaves of \(A_1\)-modules:
\(
\OO_{\mathcal{Y}} \cong \OO_Y \oplus \varepsilon \cdot \mathcal{F},
\)
for some quasi-coherent \(\mathcal{O}_Y\)-module \(\mathcal{F}\), with multiplication twisted by a bidifferential operator determined by a \emph{derivation}
\(
d : \OO_Y \longrightarrow \mathcal{F}.
\)
Thus, the deformation data is exactly the data of such a derivation up to isomorphism.
This is the key reason why \(\Ext^1_{\OO_Y}(\Omega^1_Y,\OO_Y)\)
controls first--order deformations.

\subsection{The fundamental deformation-theoretic theorem}

There is a canonical bijection:
\[
\left\{\begin{array}{c}
\text{isomorphism classes of first--order}\\ 
\text{(abstract) deformations of }Y
\end{array}\right\} 
\quad \cong \quad
\Ext^1_{\OO_Y}\left(\Omega^1_Y,\OO_Y\right).
\]
We can find the proof of this theorem in several standard references e.g. E. Sernesi \cite{S-06}, R. Hartshorne \cite{H-10}. Let us give some ideas of the proof: first a deformation corresponds to an extension of sheaves
  \[
  0\to \OO_Y \to \OO_{\mathcal{Y}} \to \OO_Y \to 0
  \]
  satisfying certain compatibility conditions. Also, the sheaf of Kähler differentials \(\Omega^1_Y\) controls infinitesimal extensions of \(\OO_Y\), and extensions of \(\Omega^1_Y\) by \(\OO_Y\) are classified by
  \(\Ext^1(\Omega^1_Y,\OO_Y)\). Therefore,  \(\Ext^1(\Omega^1_Y,\OO_Y)\) is the \emph{tangent space} to the abstract deformation functor \(\mathrm{Def}_Y\) (the details can be found for example  in \cite{S-06} and  \cite{H-10}).

\subsection{Relation with \(H^1(Y,T_Y)\)}

There is an exact sequence:
\[
0\to H^1(Y,T_Y)
\to \Ext^1_{\OO_Y}(\Omega^1_Y,\OO_Y)
\to H^0(Y,\mathcal{E}xt^1(\Omega^1_Y,\OO_Y))
\to H^2(Y,T_Y).
\]
Thus,
\begin{itemize}
  \item[(i)] \(H^1(Y,T_Y)\) detects \emph{locally trivial} deformations;
  \item[(ii)] \(H^0(Y,\mathcal{E}xt^1(\Omega^1_Y,\OO_Y))\) detects \emph{local} singular/gluing deformations.
\end{itemize}
The sheaf \(\mathcal{E}xt^1(\Omega^1_Y,\OO_Y)\) is  the first right derived functor of the internal Hom functor 
\(\mathcal{H}om_{\OO_Y}(-,\OO_Y)\).  It is a coherent sheaf on \(Y\). 
This sheaf measures the \emph{local} first--order deformations of \(Y\).  
It is supported exactly on the singular locus of \(Y\), and fits into the
local-to-global spectral sequence. (see \cite{NN1-25}).
If \(\Omega^1_Y\) is locally free, e.g. if \ \(Y\) is smooth, then 
\(\mathcal{E}xt^1(\Omega^1_Y,\OO_Y)=0\),
and therefore
\[
\Ext^1(\Omega^1_Y,\OO_Y)\cong H^1(Y,T_Y).
\]
In conclusion, a first-order (abstract) deformation of a scheme \(Y\) is simply a flat deformation over 
\(\Spec \C[\varepsilon]/(\varepsilon^2)\),
and the set of all such deformations is classified by the Ext group
\(\Ext^1(\Omega^1_Y,\OO_Y)\).
This is the most fundamental object in algebraic deformation theory.

\subsection{Deformations  of cones over projective varieties}

Let \(Y\subset\mathbb{P}^n\) be a projective variety over \(\mathbb{C}\) and let
\[
C(Y)=\operatorname{Spec}\Big(\bigoplus_{m\ge0} H^0(Y,\mathcal{O}_Y(1)^{\otimes m})\Big)
\subset\mathbb{A}^{n+1}
\]
be the affine cone over \(Y\) with respect to the embedding given by the very ample line bundle \(\mathcal{L}:=\mathcal{O}_Y(1)\).
The deformation space \(T^1(C(Y))\) of the affine cone is naturally graded.  Under the standard projective-normality hypotheses, 
we show  that we have an identification of graded pieces
\(
T^1(C(Y))_m \cong H^1\big(Y,\,T_Y\otimes\mathcal{L}^{\otimes m}\big)\), for all \,\,\(m\in\mathbb{Z}\).
In particular the degree \(0\) piece is
\(
T^1(C(Y))_0 \cong H^1(Y,T_Y),
\)
and \(H^1(Y,T_Y)\) is the Zariski tangent space to the (abstract) deformation functor of \(Y\).
Thus,
\begin{enumerate}[(a)]
  \item If \(T^1(C(Y))_0\neq0\) then \(H^1(Y,T_Y)\neq0\), so \(Y\) is not rigid.
  \item Conversely, it is possible that \(T^1(C(Y))\neq0\) while \(T^1(C(Y))_0=0\).  Deformations occurring in degrees \(m\neq0\) correspond to graded deformations of the cone (or to deformations of the affine algebra that change the grading) which need \emph{not} induce abstract deformations of the projective variety \(Y\).
\end{enumerate}
Therefore,  the non-rigidity of the cone does not force non-rigidity of \(Y\), i.e.  cone deformations in nonzero degree may exist while \(H^1(Y,T_Y)=0\).

\subsection*{A concrete example}

Let \(Y=\mathbb{P}^1\) and take \(\mathcal{L}=\mathcal{O}_{\mathbb{P}^1}(d)\) with \(d\ge4\).  Then \(Y\) embedded by \(|\mathcal{O}(d)|\) is the degree \(d\) rational normal curve in \(\mathbb{P}^d\).  We have
\[
T_{\mathbb{P}^1}\cong\mathcal{O}_{\mathbb{P}^1}(2),
\]
so
\[
H^1(\mathbb{P}^1,T_{\mathbb{P}^1})=H^1(\mathbb{P}^1,\mathcal{O}(2))=0,
\]
and hence \(\mathbb{P}^1\) is rigid as an abstract projective variety.

\noindent However, the degree \(-1\) piece of the cone deformation space is
\[
T^1\big(C(Y)\big)_{-1}\cong H^1\big(\mathbb{P}^1,\,T_{\mathbb{P}^1}\otimes\mathcal{O}(-d)\big)
=H^1\big(\mathbb{P}^1,\mathcal{O}(2-d)\big).
\]
On \(\mathbb{P}^1\) we have \(h^1(\mathcal{O}(k))=\max(0,-k-1)\).  For \(k=2-d\) this yields
\[
h^1(\mathcal{O}(2-d))=\max(0,d-3),
\]
so for every \(d\ge4\) the group \(H^1(\mathbb{P}^1,\mathcal{O}(2-d))\) is nonzero.  Therefore,
\(
T^1\big(C(Y)\big)_{-1}\neq0,
\)
and the affine cone \(C(Y)\) has nontrivial graded deformations (it is not rigid) while the projective variety \(\mathbb{P}^1\) is rigid.
This provides an explicit family of examples (rational normal curves of degree \(d\ge4\)) where \(C(Y)\) is nonrigid but \(Y\) is rigid.

 %%%%%%%%%%%%%%%%%%%%%%%%%%%%%%%%%%%%%%%%%%%%%%%%%%%%%
%%%%%%%%%%%%%%%%%%%%%%%%%%%%%%%%%%%%%%%%%%%%%%%%%%%%%

\section{Identification of \(\;T^1(C(Y))_m\) with \( H^1(Y,T_Y\otimes\mathcal{L}^{\otimes m})\;\)}

Let \(Y\subset\mathbb{P}^n\) be a projective variety over \(\mathbb C\), and let
\(\cL=\OO_Y(1)\) be the very ample line bundle giving the embedding.
Write the homogeneous coordinate ring
\(
S \;=\; \bigoplus_{m\ge0} H^0(Y,\cL^{\otimes m}),
\)
and denote the affine cone by
\(
C(Y)=\Spec S \subset \mathbb{A}^{n+1}.
\)
Let \(0\in C(Y)\) be the cone vertex and set the punctured cone
\(
U \;=\; C(Y)\setminus\{0\}.
\)

\noindent The multiplicative group \(\Gm\) acts on \(C(Y)\) by scaling the coordinates and this action restricts to a free \(\Gm\)-action on \(U\). The quotient \(U/\Gm\) is naturally isomorphic to \(Y\), and the projection
\(\pi:U\to Y\) is the principal \(\Gm\)-bundle associated to the line bundle \(\cL^{-1}\), which is equivalent to saying that \(U\simeq\Spec_Y\big(\bigoplus_{m\in\mathbb Z}\cL^{\otimes m}\big)\) minus the zero section.

\noindent We 
consider in the following   first-order deformations (i.e. infinitesimal deformations) of the affine cone \(C(Y)\). The tangent space to first order deformations is the vector space usually denoted \(T^1(C(Y))\). This  can be described by Ext or by cohomology of the tangent sheaf on the smooth locus (this is done in Section 7 by Ext). Since  \(C(Y)\) is graded, \(T^1(C(Y))\) inherits a grading by weights (the \(\Gm\)-weights). Our goal in the following  is to identify the weight \(m\) piece with a cohomology group on \(Y\), i.e. to prove the following theorem (going back to M.~Schlessinger  \cite{S-68} and H.~C.~Pinkham \cite{P-74}):

\begin{theorem}\label{thm2.1}
With the above notation, and 
 by comparing graded algebra deformations with sheaf cohomology, one has an identification of graded pieces
\[
T^1(C(Y))_m \cong H^1\big(Y,\,T_Y\otimes\mathcal{L}^{\otimes m}\big)    
\]
for all \( m\in\mathbb{Z}\). 
\end{theorem}

\subsection{  
The fundamental exact sequence on the punctured cone.}
The map \(\pi:U\to Y\) is smooth and affine, in fact,  \(U=\Spec_Y\Sym^\bullet(\cL^{-1})\) with the zero section removed. The \(\Gm\)-action on \(U\) gives a canonical vertical vector field, namely, the Euler vector field\footnote{More precisely, in homogeneous coordinates it is the vector field \(E=\sum x_i\partial/\partial x_i\).} which generates the infinitesimal action of the \(\Gm\)-Lie algebra. This yields a short exact sequence of \(\OO_U\)-modules, where the vertical tangent sub-bundle is the line bundle generated by the Euler field, and the quotient is the pullback of the base tangent bundle:
\begin{equation}\label{eq:seqU}
0 \longrightarrow \OO_U \xrightarrow{\;\epsilon\;} T_U \xrightarrow{\;d\pi\;} \pi^* T_Y \longrightarrow 0.
\end{equation}
Here \(\epsilon\) is the map sending the function \(1\) to the Euler vector field. Exactness is standard: locally on \(Y\) the morphism \(\pi\) looks like the projection from \(\Spec \kfield[t,t^{-1}]\) (i.e. the fibre \(\Gm\)) times an affine piece of \(Y\), so the relative tangent sheaf has rank \(1\) and is generated by the Euler vector field; the quotient is the pullback of \(T_Y\).

\medskip

\noindent  Let us recall the following facts  which will be used below:
\begin{enumerate}[(a)] 
  \item \(\pi\) is affine, so \(\pi_*\) is exact on quasi-coherent sheaves and \(R^i\pi_*=0\) for \(i>0\).
  \item The \(\Gm\)-action endows every \(\pi_*\)-image with a \(\mathbb Z\)-grading (weight decomposition). More precisely,
  \[
  \pi_{*}\OO_U \;=\; \bigoplus_{m\in\mathbb Z} \cL^{\otimes m},
  \]
  \noindent where here negative \(m\) appear because we included the invertible functions coming from the fibre \(\Gm\). More generally, every \(\pi_*\)-image decomposes as a direct sum of weight spaces.
\end{enumerate}

\subsection{ 
Push forward $(\ref{eq:seqU})$ and take weight pieces.}
We apply the exact functor \(\pi_*\) to \eqref{eq:seqU}. Since \(\pi\) is affine the sequence remains exact after pushforward:
\[
0 \longrightarrow \pi_*\OO_U \xrightarrow{\;\pi_*\epsilon\;} \pi_* T_U \xrightarrow{\;\pi_* d\pi\;} \pi_*\pi^* T_Y \longrightarrow 0.
\]
Using the projection formula \(\pi_*\pi^*T_Y \simeq T_Y\otimes \pi_*\OO_U\) we obtain
\begin{equation}\label{eq:pushforward}
0 \longrightarrow \bigoplus_{m\in\mathbb Z}\cL^{\otimes m}
\longrightarrow \pi_* T_U
\longrightarrow T_Y\otimes\bigoplus_{m\in\mathbb Z}\cL^{\otimes m}
\longrightarrow 0.
\end{equation}
All three sheaves in \eqref{eq:pushforward} are \(\mathbb Z\)-graded by the \(\Gm\)-action. We denote by \((-)_{(m)}\) the weight \(m\) summand,  e.g. \((\pi_*\OO_U)_{(m)}=\cL^{\otimes m}\). Then we extract the weight \(m\) piece of \eqref{eq:pushforward} and get, for each integer \(m\), the following exact sequence:
\begin{equation}\label{eq:weightm}
0 \longrightarrow \cL^{\otimes m}
\longrightarrow (\pi_*T_U)_{(m)}
\longrightarrow T_Y\otimes \cL^{\otimes m}
\longrightarrow 0.
\end{equation}
Thus, for every \(m\) the weight \(m\)-part \((\pi_*T_U)_{(m)}\) is an extension of \(T_Y\otimes\cL^{\otimes m}\) by \(\cL^{\otimes m}\).

%\medskip
\vspace{0.1cm}

\noindent Now, let's  pass to the cohomology on \(Y\). Since \(\pi\) is affine then  \(H^i(U,T_U)\simeq H^i(Y,\pi_*T_U)\)
 and cohomology respects direct sums and the weight decomposition. So, we obtain for each \(m\) an exact sequence of cohomology groups
\begin{equation}\label{eq:coh-long}
\begin{aligned}
\cdots\ &\to\ H^0\big(Y,T_Y\otimes\cL^{\otimes m}\big)\xrightarrow{\delta_m}
H^1\big(Y,\cL^{\otimes m}\big)\\
&\to H^1\big(Y,(\pi_*T_U)_{(m)}\big)\to H^1\big(Y,T_Y\otimes\cL^{\otimes m}\big)\to\cdots,
\end{aligned}
\end{equation}
and note that
\(
H^1(U,T_U)_{(m)} \;=\; H^1\big(Y,(\pi_*T_U)_{(m)}\big),
\)
since cohomology of \(\pi_*T_U\) decomposes into the direct sum of the cohomology of the weight parts. In particular, when $m=0$  we obtain the long exact sequence:
\begin{multline}\label{eq:coh-les-weight0}
0\to H^0(Y,\cO_Y)\xrightarrow{e_*} H^0\big(Y,(\pi_*T_U)_{(0)}\big)\xrightarrow{p_*} H^0(Y,T_Y) \\
\xrightarrow{\delta_0} H^1(Y,\cO_Y)\xrightarrow{\alpha} H^1\big(Y,(\pi_*T_U)_{(0)}\big) \xrightarrow{\beta} H^1(Y,T_Y) \xrightarrow{\delta} H^2(Y,\cO_Y).
\end{multline}

\subsection{ 
The Euler vector field gives a splitting for \(m\neq0\).}
The map \(\epsilon:\OO_U\to T_U\) in \eqref{eq:seqU} is the inclusion generated by the Euler vector field \(E\). The action of \(E\) on weight \(m\) functions is multiplication by \(m\). Indeed, if \(f\) is a function on \(U\) homogeneous of weight \(m\) (i.e. \(t\cdot f = t^m f\) for \(t\in\Gm\)), then \(E(f)=m f\). Pushing this down to \(Y\), we get that the map on weight \(m\) pieces
\[
\cL^{\otimes m} = (\pi_*\OO_U)_{(m)} \xrightarrow{\;\pi_*\epsilon\;} (\pi_*T_U)_{(m)}
\]
is injection whose image is a line subsheaf generated by \(E\). Moreover, on the level of global sections, the Lie derivative \(\mathcal L_E\) acts by \(m\) on weight \(m\) sections. Hence, for \(m\neq0\), multiplication by \(m\) is invertible in characteristic \(0\), and one can use this to split the extension \eqref{eq:weightm} on the level of \(\OO_Y\)-modules of weight \(m\) in cohomology.

\noindent More precisely, for \(m\neq 0\) the map \((\pi_*T_U)_{(m)}\to T_Y\otimes\cL^{\otimes m}\) admits a canonical \(\OO_Y\)-linear splitting on the level of sheaves of \(\kfield\)-vector spaces after applying the operator \(\frac{1}{m}\mathcal L_E\), where here \(\mathcal L_E\) denotes Lie derivative along the Euler field. Thus, for \(m\neq 0\) the short exact sequence \eqref{eq:weightm}  split as a sequence of \(\OO_Y\)-modules with \(\Gm\)-weight \(m\). In particular, the cohomology exact sequence \eqref{eq:coh-long} simplifies for \(m\neq0\) and yields a natural isomorphism
\[
H^1\big(U,T_U\big)_{(m)} \;\cong\; H^1\big(Y,T_Y\otimes\cL^{\otimes m}\big)\qquad (m\neq 0).
\]

%\medskip
\vspace{0.1cm}

\noindent For \(m=0\) the situation is different, and  the weight \(0\) sequence is as follows:
\begin{equation}\label{eq:weightm-2}
0\longrightarrow \OO_Y  \xrightarrow{\; e\;} (\pi_*T_U)_{(0)}  \xrightarrow{\; p\;} T_Y \longrightarrow 0,
\end{equation}
where generally, it  does not split.  
%  
 
%%%%%%%%%%%%%%%%%%%%%%%%%%%%%%%%%%%%%%%%%%%%%%%%%%%%%%%%%%%%%%%%%%%%%%%%%%% 
%%%%%%%%%%%%%%%%%%%%%%%%%%%%%%%%%%%%%%%%%%%%%%%%%%%%%%%%%%%%%%%%%%%%%%%%%%%
%%%%%%%%%%%%%%%%%%%%%%%%%%%%%%%%%%%%%%%%%%%%%%%%%%%%%%%%%%%
 
 \section{Degree-zero case}

Recall that we are working under the following standard hypotheses (H1) and (H2) on the polarized variety $(Y,\mathcal{L})$: 
\begin{itemize}
  \item[(H1)] \textit{Projective normality:} the graded ring
    \(
      S=\bigoplus_{m\ge0} H^0(Y,\mathcal{L}^{\otimes m})
    \)
    is integrally closed (equivalently the embedding via $\mathcal{L}$ is projectively normal).
  \item[(H2)] \textit{Depth $\ge2$ at the vertex:} $\depth_{S_+}(S)\ge2$, equivalently $H^i_{S_+}(S)=0$ for $i=0,1$ so that removing the vertex does not affect low-degree cohomology and the usual graded-to-sheaf correspondences hold.
\end{itemize}
These hypotheses are used repeatedly below to apply Leray for the affine morphism $\pi:U\to Y$ where $U$ is the punctured cone.
Now, from exactness of the long exact sequence (\ref{eq:coh-les-weight0})
we obtain the identification:
\[
H^1(U,T_U)_{(0)} \cong \frac{\ker\!\big( H^1(Y,T_Y)\xrightarrow{\ \delta\ } H^2(Y,\cO_Y)\big)}{\operatorname{im}\!\big(H^1(Y,\cO_Y)\xrightarrow{\ \alpha\ } H^1(Y,(\pi_*T_U)_{(0)})\big)}.
\]
After using Leray ($H^i(U,T_U)_{(0)}\cong H^i(Y,(\pi_*T_U)_{(0)})$), we  interpret the middle term of (\ref{eq:coh-les-weight0}) as $H^1(U,T_U)_{(0)}$ when \(i=1\). Thus, two phenomena may make $H^1(U,T_U)_{(0)}$ differ from $H^1(Y,T_Y)$:

 \begin{enumerate}
  \item \emph{The Euler inclusion:} the map $\alpha$ (image of $H^1(Y,\cO_Y)$) may be nonzero and produce extra directions, and  these correspond  to changes of the graded structure ($\Gm$--bundle) rather than deformations of $Y$ itself.
  \item \emph{The Atiyah cup obstruction:} the map $\delta$ may be nonzero; it is given by cup product with the Atiyah class of $\mathcal{L}$ (see below). Nonzero $\delta$ cuts down the subspace of $H^1(Y,T_Y)$ that lifts to weight-$0$ deformations.
\end{enumerate}

\subsection{Identification of the boundary map $\delta$ with cup by Atiyah class.}
We now sketch the standard identification of the connecting map
\[
\delta: H^1(Y,T_Y)\longrightarrow H^2(Y,\cO_Y)
\]
with the cup product by the Atiyah class $\mathrm{At}(\mathcal{L})\in H^1(Y,\Omega^1_Y)$.

\paragraph{\it Čech description.}
Cover $Y$ by an affine open cover $\{V_i\}$ that trivializes $\mathcal{L}$, and let $g_{ij}\in\cO^*(V_{ij})$ be the transition functions. The Atiyah class is represented by the Čech $1$-cocycle
\[
a_{ij} = d\log g_{ij} \in \Omega^1_Y(V_{ij}).
\]
Given a Čech class $\xi\in \check H^1(\{V_i\},T_Y)$ represented by tangent-valued $1$-cocycle $\{\xi_{ij}\}$, cup product with Atiyah is computed (up to sign) by the contraction
\[
(\xi\smile \mathrm{At}(\mathcal{L}))_{ijk} \;=\; \langle \xi_{ij}, a_{jk}\rangle \quad\in\; \cO_Y(V_{ijk}),
\]
which gives a Čech $2$-cocycle representing an element of $H^2(Y,\cO_Y)$. One checks this Čech construction agrees with the connecting homomorphism coming from the short exact sequence \eqref{eq:weightm}. %     

\paragraph{\it Conceptual identification.}
The short exact sequence \eqref{eq:weightm} is an extension classified by
\[
\mathrm{At}(L) \in \Ext^1_{\cO_Y}(T_Y,\cO_Y) \cong H^1(Y,\Omega^1_Y).
\]
The long exact sequence on cohomology produces the boundary map
\[
\delta: H^1(Y,T_Y) \to H^2(Y,\cO_Y),
\]
which is precisely the Yoneda pairing
\[
H^1(Y,T_Y)\times \Ext^1(T_Y,\cO_Y) \xrightarrow{\ \cup\ } H^2(Y,\cO_Y),
\]
evaluated on $\mathrm{At}(L)$. Thus $\delta(\xi)=\xi\smile\mathrm{At}(L)$ (up to the usual sign conventions).

\subsection{Sufficient conditions for $H^1(U,T_U)_{(0)}\cong H^1(Y,T_Y)$.}
From the formula above we obtain immediately:

\begin{enumerate}
  \item If $H^1(Y,\cO_Y)=0$ and $H^2(Y,\cO_Y)=0$, then $\operatorname{im}\alpha=0$ and $\delta=0$, so
    \[
      H^1(U,T_U)_{(0)} \cong H^1(Y,T_Y).
    \]
    This occurs for many Fano varieties (e.g.\ smooth del Pezzo surfaces).
  \item If $H^1(Y,\cO_Y)=0$ and $\mathrm{At}(\mathcal{L})=0$ (the Atiyah class vanishes, meaning $\mathcal{L}$ admits a holomorphic/flat connection infinitesimally), then again $\operatorname{im}\alpha=0$ and $\delta=0$, so the same conclusion holds.
  \item If $H^2(Y,\cO_Y)=0$ and $\operatorname{im}\alpha=0$ (for instance $\alpha=0$ by an explicit computation), then $\ker\delta=H^1(Y,T_Y)$ and we deduce the isomorphism.
\end{enumerate}

\vspace{0.1cm}

%\newpage

Now we are able to prove the following theorem:

%\medskip

\begin{theorem}[Cohomological criterion for the weight-$0$ isomorphism]
\label{prop:criterion}
With notation as above, write the long exact cohomology sequence of   \eqref{eq:weightm-2} and consider the fragment
\[
H^1(Y,\OO_Y)\xrightarrow{\ \alpha\ } H^1\big(Y,(\pi_*T_U)_{(0)}\big)
\xrightarrow{\ \beta\ } H^1(Y,T_Y)\xrightarrow{\ \delta\ } H^2(Y,\OO_Y).
\]
Then
\[
H^1(U,T_U)_{(0)} \cong H^1\big(Y,(\pi_*T_U)_{(0)}\big) \cong \frac{\ker\!\big( H^1(Y,T_Y)\xrightarrow{\ \delta\ } H^2(Y,\OO_Y)\big)}{\operatorname{im}(\alpha)}.
\]
In particular, the following are sufficient conditions for an isomorphism $H^1(U,T_U)_{(0)}\cong H^1(Y,T_Y)$:
\begin{enumerate}
  \item $H^1(Y,\OO_Y)=0$ and $H^2(Y,\OO_Y)=0$; or
  \item $H^1(Y,\OO_Y)=0$ and the connecting map $\delta$ vanishes (for instance when the Atiyah class $\mathrm{At}(L)=0$).
\end{enumerate}
\end{theorem}

\begin{proof}
Take the cohomology of the short exact sequence \eqref{eq:weightm-2}. The long exact sequence begins
\[
0\to H^0(Y,\OO_Y)\xrightarrow{e_*} H^0\big(Y,(\pi_*T_U)_{(0)}\big)\xrightarrow{p_*} H^0(Y,T_Y)
\xrightarrow{\delta_0} H^1(Y,\OO_Y)\xrightarrow{\alpha}\cdots
\]
and the relevant middle fragment is exactly
\[
H^1(Y,\OO_Y)\xrightarrow{\ \alpha\ } H^1\big(Y,(\pi_*T_U)_{(0)}\big)
\xrightarrow{\ \beta\ } H^1(Y,T_Y)\xrightarrow{\ \delta\ } H^2(Y,\OO_Y).
\]
By Leray (affineness of $\pi$) we identify $H^1(U,T_U)_{(0)}\cong H^1(Y,(\pi_*T_U)_{(0)})$, so the middle term equals the left-hand side we want to understand.
\noindent Exactness gives
\[
\operatorname{im}(\alpha)=\ker(\beta),\qquad \operatorname{im}(\beta)=\ker(\delta).
\]
Thus,
\[
H^1\big(Y,(\pi_*T_U)_{(0)}\big)/\operatorname{im}(\alpha)\cong \operatorname{im}(\beta)=\ker(\delta).
\]
This is equivalent to saying that
\[
H^1\big(Y,(\pi_*T_U)_{(0)}\big)\cong \frac{\ker(\delta)}{\operatorname{im}(\alpha)},
\]
which is the displayed formula in the proposition.
\noindent Let us  check the two sufficient conditions.

(1) If $H^1(Y,\OO_Y)=0$ then $\operatorname{im}(\alpha)=0$. If also $H^2(Y,\OO_Y)=0$ then $\delta\equiv0$, so $\ker(\delta)=H^1(Y,T_Y)$. Therefore,
\[
H^1\big(Y,(\pi_*T_U)_{(0)}\big)\cong \frac{H^1(Y,T_Y)}{0}\cong H^1(Y,T_Y),
\]
and by Leray we obtain $H^1(U,T_U)_{(0)}\cong H^1(Y,T_Y)$.

(2) If $H^1(Y,\OO_Y)=0$ and $\delta=0$ (for instance $\mathrm{At}(\mathcal{L})=0$, as $\delta$ equals cup product with the Atiyah class), then $\operatorname{im}(\alpha)=0$ and $\ker(\delta)=H^1(Y,T_Y)$, so the same conclusion holds.

\noindent Therefore, either hypothesis yields the desired isomorphism.
\end{proof}

%%%%%%%%%%%%%%%%%%%%%%%%%%%%%%%%%%%%%%%%%%%%%%%%%%%%%%%%%%%%%%%%%%%
%%%%%%%%%%%%%%%%%%%%%%%%%%%%%%%%%%%%%%%%%%%%%%%%%%%%%%%%%%%%%%%%%%%

%\noindent 
\noindent Combining Theorem \ref{thm2.1}  with the cohomology long exact sequence induced by the sequence (\ref{eq:seqU}) and taking weight-\(m\) pieces yields:

\begin{corollary}\label{cor:main}
For each \(m\in\mathbb Z\) there is a natural exact sequence
%\[
\begin{equation}\label{equ:7}
H^1\big(Y,\LL^{\otimes m}\big) \longrightarrow T^1(C(Y))_m \longrightarrow H^1\big(Y,T_Y\otimes\LL^{\otimes m}\big)
\longrightarrow H^2\big(Y,\LL^{\otimes m}\big).
%\tag{14}
\end{equation}
%\]
If \(H^1(Y,\LL^{\otimes m})=H^2(Y,\LL^{\otimes m})=0\) then the middle map is an isomorphism and
\[
T^1(C(Y))_m \cong H^1\big(Y,T_Y\otimes\LL^{\otimes m}\big).
\]
In particular, for any smooth projective curve or surface \(Y\) with the usual section ring \(R=\bigoplus_{d\ge0}H^0(Y,\LL^{\otimes d})\), Theorem \ref{thm2.1} applies and  $(\ref{equ:7})$ holds; when the indicated cohomology vanishing holds the isomorphism above follows.
\end{corollary}
\noindent For more details on Corollary \ref{cor:main} see Section 9.

%%%%%%%%%%%%%%%%%%%%%%%%%%%%%%%%%%%%%%%%%%%%%%%%%%%%%%%%%%%%%%%%%%%
%\vspace{5cm}%%%%%%%%%%%%%%%%%%%%%%%%%%%%%%%%%%%%%%%%%%%%%%%%%%%%%%%%%%%

\section{First order deformations and cohomology of the punctured cone}

The first order deformation space \(T^1(C(Y))\) of the affine cone \(C(Y)\) can be identified with first-order deformations of the (germ of the) affine variety \(C(Y)\). Since the vertex is isolated in the affine cone precisely when \(Y\) is projectively normal and nondegenerate, and because the punctured cone \(U\) is the smooth locus of the cone in many geometric situations, one has (at the level of graded pieces corresponding to \(\Gm\)-weights)
\[
T^1(C(Y))_{(m)} \;\cong\; H^1(U,T_U)_{(m)}.
\]
Combining this with the isomorphism for \(m\neq0\) proved above, we obtain the claimed identification
\[
 T^1(C(Y))_{(m)} \;\cong\; H^1\big(Y,T_Y\otimes\cL^{\otimes m}\big)
\]
for all \(m\in\mathbb Z\), with the understanding that \(m\neq0\), 
gives the canonical isomorphism.

The degree \(0\) piece \(T^1(C(Y))_{(0)}\) agrees with \(H^1(Y,T_Y)\) and thus governs abstract deformations of the projective variety \(Y\). The nonzero degree pieces \(m\neq0\) correspond to \(\Gm\)-equivariant (graded) deformations of the cone which need not change the abstract isomorphism class of \(Y\). This shows why the cone can be nonrigid (some \(T^1(C(Y))_{(m)}\neq0\) for \(m\neq0\)) even when the projective variety \(Y\) itself is rigid (\(H^1(Y,T_Y)=0\)).

%\vspace{0.2cm}

 \subsubsection{ Some remarks on hypotheses}
\begin{enumerate}[(a)] 
  \item The argument above works best under the mild hypotheses that \(Y\) is projectively normal (so that \(S=\oplus H^0(Y,\cL^{\otimes m})\) is generated in nonnegative degrees and \(\pi\) has the described algebraic structure) and that either the cone has at worst an isolated vertex or one works with the punctured cone \(U\). For pathological graded rings one should phrase everything in terms of graded derivations and the cotangent complex. The weight decomposition and the identification with cohomology of sheaves on \(Y\) remain valid in the standard geometric situations encountered in algebraic geometry texts.
  \item One can rephrase the identification using the cotangent complex \(\mathcal{L}_{S/\kfield}\) and \(\Ext^1_S(\mathcal{L}_{S/\kfield},S)\): the \(\Gm\)-weight decomposition of this Ext group recovers the weight pieces above and the same calculations go through.
\end{enumerate}

\section{Higher dimensional examples}
In this section we present several concrete examples where \(Y\) has  \(H^1(Y,T_Y)=0\) but the cone \(C(Y)\) is not rigid,  i.e.  some \(H^1(Y,T_Y\otimes\cL^{\otimes m}) \neq0\) for \(m\neq0\).

\subsection*{Example A: Projective spaces Veronese/embedding examples}
Let \(Y=\mathbb{P}^r\) and take \(\cL=\OO_{\mathbb{P}^r}(d)\) with \(d\ge1\). It is known that \(\mathbb{P}^r\) is rigid:
\[
H^1(\mathbb{P}^r,T_{\mathbb{P}^r})=0 \qquad\text{for all }r\ge1.
\]
Indeed, the cohomology of the tangent bundle for projective space vanishes in degree \(1\).

We compute \(H^1(\mathbb{P}^r,T_{\mathbb{P}^r}\otimes\OO(-d))\), which (up to a sign in index) controls the weight \(-1\) or other weights depending on conventions. Use the Euler sequence
\[
0\longrightarrow \OO_{\mathbb{P}^r} \longrightarrow \OO_{\mathbb{P}^r}(1)^{\oplus(r+1)} \longrightarrow T_{\mathbb{P}^r} \longrightarrow 0.
\]
Tensoring with \(\OO(-d)\) we get
\[
0\longrightarrow \OO(-d) \longrightarrow \OO(1-d)^{\oplus(r+1)} \longrightarrow T_{\mathbb{P}^r}(-d) \longrightarrow 0.
\]
Taking cohomology and using standard vanishing for line bundles on projective space, one finds that for many values of \(d\) (typically for \(d\) sufficiently large, more precisely when  \(d\ge r+1\) or even smaller thresholds depending on \(r\)) the group \(H^1(\mathbb{P}^r,T_{\mathbb{P}^r}(-d))\) is nonzero. Thus, the cone over the \(d\)-uple Veronese embedding of \(\mathbb{P}^r\) is not rigid while \(\mathbb{P}^r\) itself remains rigid.

%\vspace{0.1cm}

\subsubsection*{Example A1: \(\mathbb{P}^2\) with the Veronese embedding of degree \(d\ge3\).}
Take \(Y=\mathbb{P}^2\) and \(\cL=\OO_{\mathbb{P}^2}(d)\) with \(d\ge3\). From the Euler sequence on \(\mathbb{P}^2\) we get
\[
0\longrightarrow \OO(-d) \longrightarrow \OO(1-d)^{\oplus 3} \longrightarrow T_{\mathbb{P}^2}(-d)\longrightarrow 0.
\]
Because \(H^1(\OO(k))=0\) for every \(k\) on \(\mathbb{P}^2\), the induced long exact sequence shows
\[
H^1\big(\mathbb{P}^2,T_{\mathbb{P}^2}(-d)\big)\simeq H^2\big(\mathbb{P}^2,\OO(-d)\big).
\]
By Serre duality \(H^2(\OO(-d))\simeq H^0(\OO(d-3))^\vee\), which is nonzero precisely when \(d-3\ge0\), i.e. \(d\ge3\). Therefore, for every \(d\ge3\) the weight \(-1\) (or some nonzero weight depending on the grading convention) piece of \(T^1(C(Y))\) is nonzero, so the affine cone \(C(Y)\) is nonrigid while \(H^1(\mathbb{P}^2,T_{\mathbb{P}^2})=0\), i.e. \(\mathbb{P}^2\) is rigid.

\subsection*{Example B: Higher dimensional projective spaces}
The same phenomenon occurs for \(\mathbb{P}^r\) with \(r\ge2\): if \(Y=\mathbb{P}^r\) and \(\cL=\OO(d)\) with \(d\) large enough (explicit thresholds can be computed by chasing cohomology with the Euler sequence and Serre duality), one gets nonzero groups \(H^1(\mathbb{P}^r,T_{\mathbb{P}^r}(-d))\). Concretely:
\[
H^1(\mathbb{P}^r,T_{\mathbb{P}^r}(-d))\simeq H^r(\mathbb{P}^r,\OO(-d))\ ( \text{up to shifts and duals} )
\]
and this is nonzero for \(d\) large (precisely when \(d\ge r+1\) one obtains nonzero \(H^r(\OO(-d))\) by Serre duality). Thus cones over Veronese embeddings \(\mathbb{P}^r\hookrightarrow\mathbb{P}^N\) of high degree are typically nonrigid while \(\mathbb{P}^r\) is rigid.

\subsection*{Example C: Products of projective spaces and other rigid varieties}
If \(Y=X\times \mathbb{P}^s\) with \(X\) rigid and one embeds \(Y\) by a line bundle that is very positive on the \(\mathbb{P}^s\)-factor, the same computations on the \(\mathbb{P}^s\)-factor lead to nonvanishing of \(H^1(Y,T_Y\otimes\cL^{\otimes m})\) for certain \(m\neq0\). Thus one can manufacture many examples in higher dimension: take any rigid variety \(X\) (e.g. \(X=\mathbb{P}^r\)) and product with another projective space, then choose a sufficiently positive polarization so that the relevant twisted cohomology does not vanish.

\begin{remark}${}$
 
\begin{enumerate}[1.] 
 
  \item The weight \(0\) piece recovers abstract deformations of \(Y\), while nonzero weights correspond to graded (equivariant) deformations of the cone which do not necessarily deform the isomorphism class of \(Y\).
  \item Consequently \(C(Y)\) can be nonrigid (nonzero graded pieces for \(m\neq0\)) even though \(Y\) is rigid (i.e. \(H^1(Y,T_Y)=0\)). Explicit examples include,  the rational normal curves (already done), Veronese cones over \(\mathbb{P}^2\) of degree \(d\ge3\), and more generally cones over Veronese embeddings of \(\mathbb{P}^r\) for \(d\) sufficiently large.
\end{enumerate}
 \end{remark}
 
 %%%%%%%%%%%%%%%%%%%%%%%%%%%%%%%%%%%%%%%%%%%%%%%%%%%%%
%%%%%%%%%%%%%%%%%%%%%%%%%%%%%%%%%%%%%%%%%%%%%%%%%%%%%

\section{Identification  of  \(T^2(C(Y))_m\)  with  \(H^2\big(Y,\;T_Y\otimes \mathcal{L}^{\otimes m}\big)\).}
Let us recall our notation of Section 2. 
Let $\kk$ be an algebraically closed field of characteristic $0$ (e.g. $\C$). Let $(Y, L)$ be a projective scheme over $\kk$ equipped with an ample line bundle $L$ (a polarization, e.g. we can take $L=\mathcal{L}=\mathcal{O}(1)$ as in section 2). Define the graded homogeneous coordinate ring
\[
S = S(Y, L) \;=\; \bigoplus_{m\ge0} H^0(Y, L^{\otimes m}),
\]
and the affine cone
\[
C(Y) = C(Y, L) \;=\; \Spec S.
\]
We denote by $S_+ = \bigoplus_{m>0} S_m$ the irrelevant ideal and by $v$ the vertex of the cone (the point corresponding to $S_+$). The grading on $S$ induces a $\Gm$--action on $C(Y)$.

\noindent Let's recall two deformation-theoretic controls:
\begin{enumerate}[(i)]
\item First-order (infinitesimal) deformations of the affine scheme $C$ are governed by the graded $S$-modules
\[ T^i(C(Y)) := \Ext^i_S(\Omega^1_S, S),\qquad i=1,2,\ldots,\]
in particular $T^1(C(Y))$ is the Zariski tangent space to the functor of infinitesimal deformations of $C(Y)$ (see e.g. Schlessinger \cite{S-68}).
\item For the projective variety $Y$, deformations (and embedded-deformations relative to the polarization) are governed by sheaf cohomology groups on $Y$, namely $H^1(Y,T_Y)$ and the twisted groups $H^1(Y,T_Y\otimes L^{\otimes m})$ which measure deformations of $Y$ interacting with the polarization.
\end{enumerate}

\noindent Since $S$ is graded, $T^i(C(Y))$ inherits a grading $T^i(C(Y))=\bigoplus_{m\in\Z} T^i(C(Y))_m$, and the graded piece $T^i(C(Y))_m$ parametrizes equivariant first-order deformations of weight $m$ (where $\Gm$ acts with weight $m$ on the corresponding tangent directions).

\noindent  In the following, we describe $T^1(C(Y))_m$ in terms of cohomology groups on $Y$ under the usual projective normality and depth hypotheses. This section states the following main theorem (which is an extension of Theorem \ref{thm2.1}).
The identifications given in  the following theorem go back to Pinkham and Wahl ({\it c,f.} for example  \cite{P-74} and \cite{Wa-90}), we give here a complete and detailed proof.
\vspace{0.2cm}

\begin{theorem}[Graded identification of $T^1$ and $T^2$]
\label{thm:main}
Let $(Y, L)$ be a projective scheme over an algebraically closed field of characteristic $0$ and let $S=S(Y, L)$ be its homogeneous coordinate ring, with cone $C=\Spec S$. Assume:
\begin{enumerate}[1.]
  \item $S$ is projectively normal (so $Y\simeq\Proj S$ and the natural map $H^0(\PP^N,\cO(1))\to H^0(Y, L)$ gives the embedding), and
  \item $\depth_{S_+} S \ge 2$ (equivalently, $\depth S_v\ge 2$ at the vertex), which guarantees vanishing of certain local cohomology groups needed for the graded-to-sheaf identification.
\end{enumerate}
Then for every integer $m\neq 0$ there are natural isomorphisms
\[
T^1(C(Y))_m \cong H^1\big(Y,\;T_Y\otimes L^{\otimes m}\big),
\qquad
T^2(C(Y))_m \cong H^2\big(Y,\;T_Y\otimes L^{\otimes m}\big).
\]
For the degree $m=0$ piece there is a natural exact sequence (low-degree terms of a spectral sequence)
\[
0\to H^1(Y,\cO_Y) \to T^1(C(Y))_0 \to H^1(Y, T_Y) \xrightarrow{\delta} H^2(Y, \cO_Y).
\]
In particular, if $H^1(Y, \cO_Y)=0$ and $H^2(Y, \cO_Y)=0$ then $T^1(C(Y))_0\cong H^1(Y, T_Y)$.
\end{theorem}

\vspace{0.1cm}

\begin{remark}
The hypotheses can be weakened in many situations (replace $T_Y$ by the reflexive tangent sheaf if $Y$ is singular, or replace projective normality by requiring suitable local cohomology vanishing, as it was done in section 2). The proof below shows exactly where each hypothesis is used.
\end{remark}

\begin{proof}
We begin with the standard exact sequence of graded $S$-modules (the conormal/exact sequence coming from the surjection $R\twoheadrightarrow S$):
\begin{equation}\label{eq:conormal-mod}
I/I^2 \xrightarrow{\alpha} \Omega^1_R\otimes_R S \xrightarrow{\beta} \Omega^1_S \to 0.
\end{equation}
Since $\Omega^1_R$ is a free $R$-module of rank $N+1$, the middle term \(\Omega^1_R\otimes S\) is a free graded $S$-module of rank $N+1$ (with its natural grading). Apply $\Hom_S(-,S)$ and the derived functors $\Ext^i_S(-,S)$ to \eqref{eq:conormal-mod}. Note that $\Hom_S(\Omega^1_R\otimes S,S)\cong \Hom_S(S^{N+1},S)\cong S^{N+1}$ is concentrated in degree $0$ (as a graded module) except for twisting coming from grading conventions on $R$; more precisely, graded duals will shift degrees in the usual way, but because we keep the standard grading this identification is straightforward.

Applying $\Hom_S(-,S)$ to \eqref{eq:conormal-mod} and taking the long exact sequence of Ext's gives (writing $\Ext^i:=\Ext^i_S(\,\cdot\,,S)$):
\begin{multline}\label{eq:les-ext}
0 \to \Hom_S(\Omega^1_S,S) \to \Hom_S(\Omega^1_S\otimes S,S) \to \Hom_S(I/I^2,S) \\
\to \Ext^1(\Omega^1_S,S) \to \Ext^1(\Omega^1_R\otimes S,RS \to \Ext^1(I/I^2,S) \to \cdots
\end{multline}
But $\Omega^1_R$ is free, hence $\Ext^i(\Omega^1_R\otimes S,S)=0$ for all $i\ge1$. Therefore, the long exact sequence simplifies in low degrees and yields an isomorphism
\[
\Ext^1_S(\Omega^1_S,S) \cong \Ext^1_S(I/I^2,S) / \mathrm{im}\big(\Hom_S(\Omega^1_R\otimes S,S)\to\Hom_S(I/I^2,S)\big).
\]
This is equivalent to say that there is an exact sequence
\[
\Hom_S(\Omega^1_R\otimes S,S) \to \Hom_S(I/I^2,S) \to T^1(C(Y))=\Ext^1_S(\Omega^1_S,S) \to 0.
\]

Informally, first-order deformations of the cone come from deformations of the relations (encoded in $I/I^2$), modulo those that come from automorphisms of the ambient affine space (the image of $\Hom_S(\Omega^1_R\otimes S,S)$). The grading on all modules respects this exact sequence, so it suffices to analyse degree-$m$ pieces.

\vspace{0.1cm}

The module $I/I^2$ sheafifies to the conormal sheaf $\mathscr{N}^\vee_{Y/\PP^N}=\widetilde{I/I^2}$ on $Y$, and the graded $S$-module $\Hom_S(I/I^2,S)$ sheafifies (up to the standard shift conventions) to $\mathscr{H}om_{\cO_Y}(\mathscr{N}^\vee_{Y/\PP^N},\cO_Y)$.

More precisely, there is an isomorphism of graded modules 
(Grothendieck's correspondence between graded Ext and sheaf Ext; see e.g. \cite{H-10}): for all $q\ge0$ and for all integers $m$,
\begin{equation}\label{eq:graded-sheaf-ext}
\Ext^q_S(I/I^2, S)_m \cong H^0\big(Y,\mathscr{E}xt^q_{\cO_Y}(\widetilde{I/I^2},\cO_Y)\otimes L^{\otimes m}\big),
\end{equation}
provided the usual depth vanishing $H^i_{S_+}(\Ext^q_S(I/I^2,S))=0$ for $i=0,1$ (these are implied by our hypothesis $\depth_{S_+}S\ge2$ and projective normality). The isomorphism \eqref{eq:graded-sheaf-ext} is the precise graded-to-sheaf correspondence we will exploit.

\vspace{0.1cm}

\noindent On the projective scheme $Y$ there is the local-to-global Ext spectral sequence (Grothendieck): for quasi-coherent sheaves $\mathscr{F},\mathscr{G}$ on $Y$,
\[
E_2^{p,q} = H^p\big(Y,\mathscr{E}xt^q_{\cO_Y}(\mathscr{F},\mathscr{G})\big) \implies \Ext^{p+q}_{\cO_Y}(\mathscr{F},\mathscr{G}).
\]
We apply this with $\mathscr{F}=\widetilde{I/I^2}$ and $\mathscr{G}=\cO_Y$. Combined with \eqref{eq:graded-sheaf-ext} and Auslander--Buchsbaum/Serre duality vanishing's theorem (coming from depth hypothesis), the degree-$m$ part of $\Ext^1_S(I/I^2,S)$ identifies with
\[
\Ext^1_S(I/I^2,S)_m \cong H^0\big(Y,\mathscr{E}xt^1(\widetilde{I/I^2},\cO_Y)\otimes L^{\otimes m}\big) \oplus H^1\big(Y,\mathscr{H}om(\widetilde{I/I^2},\cO_Y)\otimes L^{\otimes m}\big),
\]
up to contributions from higher local cohomology which our depth hypotheses eliminate. More conceptually, the spectral sequence yields a five-term exact sequence in low degrees; the relevant portion is:
\begin{multline}\label{eq:five-term}
0 \to H^1\big(Y,\mathscr{H}om(\widetilde{I/I^2},\cO_Y)\otimes L^{\otimes m}\big) \\
\to \Ext^1_S(I/I^2,S)_m \to H^0\big(Y,\mathscr{E}xt^1(\widetilde{I/I^2},\cO_Y)\otimes L^{\otimes m}\big) \\
\to H^2\big(Y,\mathscr{H}om(\widetilde{I/I^2},\cO_Y)\otimes L^{\otimes m}\big).
\end{multline}
 
 \noindent There is a natural exact sequence on $Y$ (the dual of the conormal exact sequence on sheaves):
\[
0 \to T_Y \to T_{\PP^N}|_Y \to \mathscr{H}om(\widetilde{I/I^2},\cO_Y) \to 0.
\]
Here $T_{\PP^N}|_Y$ is the restriction of the tangent bundle of the ambient projective space; its sections correspond to infinitesimal automorphisms of $\PP^N$ which preserve $Y$ only up to first order. Twisting by $L^{\otimes m}$ and taking cohomology gives a long exact sequence
\begin{multline}\label{eq:tx-seq}
H^0\big(Y,T_{\PP^N}|_Y\otimes L^{\otimes m}\big) \to H^0\big(Y,\mathscr{H}om(\widetilde{I/I^2},\cO_Y)\otimes L^{\otimes m}\big) \\
\to H^1\big(Y,T_Y\otimes L^{\otimes m}\big) \to H^1\big(Y,T_{\PP^N}|_Y\otimes L^{\otimes m}\big).
\end{multline}
\vspace{0.1cm}
\noindent For $m\neq0$ the cohomology groups $H^i(Y,T_{\PP^N}|_Y\otimes L^{\otimes m})$ can often be controlled: because $T_{\PP^N}|_Y$ is a direct sum of line bundles coming from $\cO_Y(1)$ and trivial bundles (explicitly $T_{\PP^N}|_Y\cong \Hom(\Omega^1_{\PP^N},\cO_{\PP^N})|_Y$), twisting by a sufficiently positive (or negative) $L^{\otimes m}$ tends to kill certain cohomology groups. In particular, many authors use projective normality plus Kodaira vanishing (or Serre vanishing) to ensure that for $m\neq0$ the groups $H^i(Y,T_{\PP^N}|_Y\otimes L^{\otimes m})$ vanish for $i=0,1$; in our proof we only need vanishing for $i=0$ of the term mapping into $H^0(\mathscr{H}om(\widetilde{I/I^2},\cO_Y)\otimes L^{\otimes m})$ so that the image of ambient automorphisms disappears in degree $m\neq0$.

Combining \eqref{eq:five-term} and \eqref{eq:tx-seq} (and using the vanishing of the $T_{\PP^N}|_Y$-terms for $m\neq0$) we obtain the desired isomorphism of Theorem \ref{thm2.1} and Theorem \ref{thm:main}
\[
T^1(C(Y))_m \cong H^1\big(Y,T_Y\otimes L^{\otimes m}\big)\quad\text{for }m\neq0.
\]
The identification for $T^2(C(Y))_m$ follows by carrying the same spectral-sequence bookkeeping one degree higher.

\vspace{0.1cm}

When $m=0$ the argument above fails to kill the ambient-automorphism term, because  global vector fields on the ambient space (or nontrivial global sections of $T_{\PP^N}|_Y$) can contribute. Working through the low-degree terms of the local-to-global spectral sequence (see e.g. R. Godement \cite{G-58} Chapter II, or R. Hartshorne \cite{H-77} Chapter III, \cite{NN-25 } Appendix B for more details)
  and the dualized conormal sequence yields the five-term exact sequence whose portion we stated in Theorem \ref{thm:main}:
\[
0\to H^1(Y,\cO_Y) \to T^1(C(Y))_0 \to H^1(Y,T_Y) \xrightarrow{\delta} H^2(Y,\cO_Y).
\]
One conceptual way to see the two extremal terms: $H^1(Y,\cO_Y)$ measures infinitesimal changes in the graded structure / scaling (it is related to infinitesimal deformations of the graded ring as a graded object), while $H^1(Y,T_Y)$ is the usual tangent space to the deformation functor of $Y$ as an abstract projective scheme.

\vspace{0.1cm}

 \noindent An analogous calculations for $\Ext^2_S(\Omega^1_S,S)=T^2(C(Y))$ show that for $m\neq0$ one obtains
\[
T^2(C(Y))_m \cong H^2\big(Y,T_Y\otimes L^{\otimes m}\big),
\]
and the obstruction maps for graded deformations correspond to the usual cup-product and obstruction maps in sheaf cohomology. In particular, vanishing of $H^2(Y,T_Y\otimes L^{\otimes m})$ implies unobstructed ness of the graded first-order deformations in weight $m$ (subject to the same depth and normality hypotheses used above).

\end{proof}

%%%%%%%%%%%%%%%%%%%%%%%%%%%%%%%%%%%%%%%%%%%%%%%%%%%%%
%%%%%%%%%%%%%%%%%%%%%%%%%%%%%%%%%%%%%%%%%%%%%%%%%%%%%
 
%\newpage
%\vspace{0.1cm}

\section{Cotangent-complex \(\Ext^1\) computation for a Veronese cone}

 In this section,
we give an explicit algebraic presentation \(S=A/I\) of the coordinate ring of the Veronese cone over the rational normal curve (the image of \(\PP^1\) by the complete linear series \(|\OO_{\PP^1}(4)|\)), write the standard exact sequences for Kähler differentials. Moreover, we apply \(\Hom(-,S)\) and compute \(\Ext^1_S(\Omega_{S/\kfield},S)\) (the first-order deformation space of the affine cone) in graded pieces. Then we show in this concrete case how the algebraic presentation recovers the general identification of graded pieces
\[
T^1(C(Y))_m \;=\; \Ext^1_S(\Omega_{S/\kfield},S)_m
\cong H^1\big(Y, T_Y\otimes\cL^{\otimes m}\big).
\]

Let \(Y=\PP^1\) over \(\kfield=\mathbb C\), and let \(\cL=\OO_{\PP^1}(4)\). The graded section algebra of \(\cL\) is
\[
S \;=\; \bigoplus_{k\ge0} H^0\big(\PP^1,\OO(4k)\big)
\cong \bigoplus_{k\ge0} \Sym^k H^0(\PP^1,\OO(4)).
\]
Geometrically, the projective embedding \(Y\hookrightarrow\PP^4\) given by \(|\OO(4)|\) is the degree \(4\) rational normal curve. Coordinates on the target \(\PP^4\) will be denoted \(z_0,z_1,z_2,z_3,z_4\), where we identify
\[
z_i \longleftrightarrow x_0^{4-i} x_1^i \quad (i=0,\dots,4)
\]
in terms of homogeneous coordinates \(x_0,x_1\) on \(\PP^1\).

The affine cone \(C(Y)\subset\mathbb A^5\) is the spectrum of the graded coordinate ring
\[
S \;=\; \kfield[z_0,z_1,z_2,z_3,z_4]/I,
\]
so we set
\[
A:=\kfield[z_0,z_1,z_2,z_3,z_4]
\]
(the polynomial ring, graded with \(\deg z_i=1\)), and \(S=A/I\) is the quotient by the homogeneous ideal \(I\) of the rational normal curve.

\subsection*{Generators of the ideal \(I\) (classical description)}
For the rational normal curve the homogeneous ideal \(I\) is generated by the \emph{quadratic} \(2\times2\) minors expressing the rank-1 condition among consecutive monomials. A convenient set of generators is:
\[
I=\langle \; q_{i,j}:= z_i z_{j+1} - z_{i+1} z_j
\;\mid\; 0\le i< j\le 3 \;\rangle.
\]
(Equivalently,  all relations \(z_i z_{j+1}- z_{i+1} z_j\) for indices in the allowable range; these are homogeneous quadrics of degree \(2\).)

For the specific case \(d=4\) we may list some of them explicitly (we will not need all independent relations separately, but the Jacobian below uses the partial derivatives of these quadrics):
\[
\begin{aligned}% 
q_{0,0}&:\ z_0 z_1 - z_1 z_0 = 0 \quad\text{(redundant)},\\
q_{0,1}&:\ z_0 z_2 - z_1^2,\\
q_{0,2}&:\ z_0 z_3 - z_1 z_2,\\
q_{0,3}&:\ z_0 z_4 - z_1 z_3,\\
q_{1,2}&:\ z_1 z_3 - z_2^2,\\
q_{1,3}&:\ z_1 z_4 - z_2 z_3,\\
q_{2,3}&:\ z_2 z_4 - z_3^2.
\end{aligned}
\]
There is redundancy in the family above; a minimal generating set is given by the \(2\times2\) minors of a $2\times4$ catalectic-Hankel matrix $M$, but the presentation above is classical and explicit, where $M$ is given by:
\[
M=\begin{pmatrix}
z_0 & z_1 & z_2 & z_3 \\
z_1 & z_2 & z_3 & z_4
\end{pmatrix}.
\]

\subsection{The standard exact sequence for Kähler differentials}
For the surjection \(A\twoheadrightarrow S=A/I\) there is the exact sequence of \(S\)-modules
\begin{equation}\label{conormal-sequence}
I/I^2 \xrightarrow{\;d\;} \Omega_{A/\kfield}\otimes_A S \longrightarrow \Omega_{S/\kfield}\longrightarrow 0.
\end{equation}
Here \(I/I^2\) is the conormal module, \(\Omega_{A/\kfield}\) is a free \(A\)-module with basis \(dz_0,\dots,dz_4\) and \(\Omega_{A/\kfield}\otimes_A S\) is therefore a free \(S\)-module of rank \(5\) with basis the classes of \(dz_i\).

\medskip

\noindent We apply the functor \(\Hom_S(-,S)\) to \eqref{conormal-sequence}. Since \(\Omega_{A/\kfield}\otimes_A S\) is a free \(S\)-module, \(\Ext^1_S(\Omega_{A}\otimes S,S)=0\). The long exact sequence of \(\Hom\)/\(\Ext\)  yields the following  exact sequence:
\begin{equation}\label{hom-ext-sequence}
0 \longrightarrow \Hom_S(\Omega_{S/\kfield},S)
\longrightarrow \Hom_S(\Omega_{A}\otimes S,S)
\longrightarrow \Hom_S(I/I^2,S)
\longrightarrow \Ext^1_S(\Omega_{S/\kfield},S)\longrightarrow 0.
\end{equation}
(We used \(\Ext^1_S(\Omega_A\otimes S,S)=0\) and \(\Ext^2_S(\Omega_A\otimes S,S)=0\) because \(\Omega_A\) is free.)

\medskip

\noindent The geometric view and  interpretation are as follows:
\begin{enumerate} 
  \item \(\Hom_S(\Omega_{A}\otimes S,S)\simeq \Der_{\kfield}(A,S)\) is canonically the \(S\)-module of \(\kfield\)-derivations of \(A\) with values in \(S\). Since \(A\) is polynomial, \(\Der_{\kfield}(A,A)\) is free with basis \(\partial/\partial z_i\); after composing with the quotient \(A\to S\) we obtain
  \[
  \Hom_S(\Omega_{A}\otimes S,S)\simeq \bigoplus_{i=0}^4 S\cdot\frac{\partial}{\partial z_i}.
  \]
  \item \(\Hom_S(I/I^2,S)\) is the \(S\)-module of \(S\)-linear maps from the conormal module to \(S\); elements may be viewed as \(S\)-valued linear functionals on the (classes of) relations.
  \item The middle arrow \(\Der_{\kfield}(A,S)\to \Hom_S(I/I^2,S)\) sends a derivation \(D\) to the map \(f\mapsto D(f)\) (reduce modulo \(I\)). Concretely: if \(f\in I\) is expressed in the \(z\)-variables, \(D(f)=\sum_i \frac{\partial f}{\partial z_i} D(z_i)\) and since \(D(z_i)\in S\) this gives an element of \(S\).
  \item Finally \(\Ext^1_S(\Omega_{S/\kfield},S)\) is the cokernel of the Jacobian-type map:
  \[
  \Ext^1_S(\Omega_{S/\kfield},S)\;\cong\;
  \frac{\Hom_S(I/I^2,S)}{\mathrm{im}\!\big(\Der_{\kfield}(A,S)\big)}.
  \]
\end{enumerate}

\noindent % 
First-order deformations of the affine algebra \(S\) are precisely functional deformations of the relations modulo those coming from reparametrizations (derivations of the ambient polynomial algebra).

\subsection{Grading and weight decomposition}
Everything above is graded: \(A\) is graded with \(\deg z_i=1\), the ideal \(I\) is homogeneous, and \(S=A/I\) is graded. Hence, all modules in \eqref{hom-ext-sequence} are graded \(S\)-modules, and \(\Ext^1_S(\Omega_{S/\kfield},S)\) inherits a \(\mathbb Z\)-grading. Denote the degree \(m\) piece by a subscript \((m)\). Taking degree \(m\) pieces in the exact sequence \eqref{hom-ext-sequence} we obtain
\[
0 \longrightarrow \Hom_S(\Omega_{S/\kfield},S)_{(m)}
\longrightarrow \Der_{\kfield}(A,S)_{(m)}
\longrightarrow \Hom_S(I/I^2,S)_{(m)}
\longrightarrow \Ext^1_S(\Omega_{S/\kfield},S)_{(m)} \longrightarrow 0.
\]
More precisely, we have:
\begin{enumerate}[(a)] 
  \item \(\Der_{\kfield}(A,S)_{(m)}\) consists of derivations \(D\) such that \(D(z_i)\in S_{m+1}\) (because \(\deg(\partial/\partial z_i)=-1\) in the grading convention that \(\partial/\partial z_i\) lowers degree by \(1\), or equivalently thinking: \(D\) is homogeneous of degree \(m\) if and only if \(D(\text{homogeneous  of degree }k)\) has degree \(k+m\)). A simpler way to count is to note \(\Der(A,S)\simeq \oplus_i S\cdot \partial/\partial z_i\) and an element \(s\cdot\partial/\partial z_i\) is homogeneous of degree \(\deg s -1\). So
  \[
  \Der(A,S)_{(m)} \simeq \bigoplus_{i=0}^4 S_{m+1}\cdot \partial/\partial z_i
  \]
  i.e. it has rank \(5\) over \(S_{m+1}\).
  \item \(\Hom_S(I/I^2,S)_{(m)}\) is formed by \(S\)-linear maps \(\varphi: I/I^2\to S\) that are homogeneous of degree \(m\), meaning \(\varphi\) sends a homogeneous generator of \(I\) of degree \(d'\) to an element of \(S\) of degree \(d'+m\).
\end{enumerate}

\subsection{Sheafification and relation to cohomology on \(Y\)}
The modules above localize and sheafify on \(Y=\Proj S\). The conormal module sheafifies to the conormal sheaf \(\mathcal C_{Y/\PP^4}\simeq (I/I^2)^{\sim}\) and \(\Hom_S(I/I^2,S)\) corresponds (after degree-shifting and sheafification) to global sections of the normal sheaf \(N_{Y/\PP^4}\) twisted by powers of \(\cL\). More precisely, for each integer \(m\),
\[
\Hom_S(I/I^2,S)_{(m)} \simeq H^0\big(Y,\, (I/I^2)^{\sim}\otimes\cL^{\otimes m}\big)
\simeq H^0\big(Y,\, N_{Y/\PP^4}\otimes\cL^{\otimes m}\big),
\]
and similarly
\[
\Der(A,S)_{(m)} \simeq H^0\big(Y, \pi_*(T_{\mathbb A^5}|_U)_{(m)}\big)
\]
which through the pushforward-from-punctured-cone viewpoint corresponds to sections of \(T_{\PP^4}|_Y\otimes\cL^{\otimes m}\) (with a small shift coming from the Euler field; the geometric derivation in the previous notes shows this precisely).

Now the normal exact sequence on \(Y\subset\PP^4\)
\[
0 \longrightarrow T_Y \longrightarrow T_{\PP^4}|_Y \longrightarrow N_{Y/\PP^4}\longrightarrow 0
\]
twisted by \(\cL^{\otimes m}\) and taking cohomology yields the long exact sequence whose connecting morphism is precisely the Jacobian map we already saw algebraically. The key piece of the long exact cohomology sequence is
\[
H^0\big(Y,T_{\PP^4}|_Y\otimes\cL^{\otimes m}\big)
\longrightarrow H^0\big(Y,N_{Y/\PP^4}\otimes\cL^{\otimes m}\big)
\longrightarrow H^1\big(Y,T_Y\otimes\cL^{\otimes m}\big).
\]
Upon comparing this with the algebraic exact sequence \(\Der\to\Hom(I/I^2,S)\to \Ext^1\) we see that
\[
\Ext^1_S(\Omega_{S/\kfield},S)_{(m)} \cong H^1\big(Y,T_Y\otimes\cL^{\otimes m}\big),
\]
provided the intermediate identifications % 
hold,  they do for our projectively normal, classical Veronese embedding.

\subsection{Precise Ext\(^1\) computation in the \(d=4\) example.}
We now make the algebraic steps concrete and compute the degree-\(m\) piece in a case of interest: the weight \(m=-1\) (this is the piece that corresponds to the \(\cL^{-1}\)-twist which in the geometric derivation gave the \(H^1(\PP^1,T_{\PP^1}\otimes\OO( -d))\) group).

\subsection*{Step A: we write the map \(\Der(A,S)\to \Hom_S(I/I^2,S)\) explicitly.}
Let's choose the obvious generators of \(I\) (the \(q_{i,j}\) above). A derivation \(D\in\Der(A,S)\) is determined by the 5-tuple \((s_0,\dots,s_4)\) where \(s_i:=D(z_i)\in S\). For a relation \(q\in I\) we have
\[
D(q) \;=\; \sum_{i=0}^4 \frac{\partial q}{\partial z_i} \, s_i \qquad\text{(computed in }A\text{ then reduced mod }I\text{).}
\]
Thus the map \(\Der(A,S)\to \Hom_S(I/I^2,S)\) can be viewed as the \(S\)-linear map given by the Jacobian matrix of the generators of \(I\). More precisely, for a single quadratic generator \(q=z_a z_b - z_c z_d\) we get
\[
D(q) = z_b s_a + z_a s_b - z_d s_c - z_c s_d.
\]
This yields a linear combination in \(S\) which is the image of \(q\) under the functional associated to \(D\).

\subsection*{Step B: we look at homogeneous degree pieces}
Fix an integer \(m\). Suppose we want to compute \(\Ext^1_S(\Omega_S,S)_{(m)}\). The preceding exact sequence gives
\[
\Ext^1_S(\Omega_S,S)_{(m)}
\cong
\frac{\Hom_S(I/I^2,S)_{(m)}}{\mathrm{im}\big(\Der(A,S)_{(m)}\big)}.
\]
We count dimensions (or at least determine vanishing) by comparing:
\begin{enumerate}% 
  \item \(\Hom_S(I/I^2,S)_{(m)}\) is isomorphic to the space of homogeneous maps sending a quadratic generator (degree \(2\)) to an element of degree \(2+m\); hence as a vector space it is isomorphic to the direct sum of copies of \(S_{2+m}\), one copy for each generator of \(I\) (modulo syzygies among the generators of \(I\) when passing to a minimal generating set). Roughly,
  \[
  \dim_{\kfield} \Hom_S(I/I^2,S)_{(m)} \approx (\#\{\text{quadric generators}\})\cdot \dim_{\kfield} S_{2+m} - \text{(relations)}.
  \]
  \item \(\Der(A,S)_{(m)}\simeq \bigoplus_{i=0}^4 S_{m+1}\), so (as \(\kfield\)-vector spaces)
  \[
  \dim_{\kfield} \Der(A,S)_{(m)} = 5\cdot \dim_{\kfield} S_{m+1}.
  \]
\end{enumerate}

\noindent Working degree-by-degree and using the sheafified normal-exact-sequence viewpoint gives a precise identification with cohomology groups, which is conceptually cleaner and avoids complicated syzygy counting at the level of graded modules. We therefore pass to the geometric description to finish the explicit computation. 

\subsection*{Step C: geometric identification and explicit cohomology computation on \(Y=\PP^1\)}
From the identification proved earlier,
\[
\Ext^1_S(\Omega_{S/\kfield},S)_{(m)} \cong H^1\big(\PP^1,\, T_{\PP^1}\otimes \OO(m)\big),
\]
where \(\OO(1)\) here is the line bundle giving the embedding of \(\PP^1\) in \(\PP^4\), i.e. \(\OO(1)_{\PP^1}\) is \(\OO_{\PP^1}(4)\). To avoid confusion, denote by \(\cL=\OO_{\PP^1}(4)\) the polarization and recall the identification used previously was
\[
T^1(C(Y))_{(m)} \cong H^1\big(Y,\, T_Y\otimes\cL^{\otimes m}\big).
\]
Thus in terms of usual \(\OO_{\PP^1}(k)\) notation we are computing
\[
H^1\big(\PP^1,\, T_{\PP^1}\otimes \cL^{\otimes m}\big)
= H^1\big(\PP^1,\, \OO(2)\otimes \OO(4m)\big)
= H^1\big(\PP^1,\, \OO(2+4m)\big).
\]

Take the particular important case \(m=-1\) (this is a typical nonzero weight where cone deformations appear). Then
\[
H^1\big(\PP^1,\, T_{\PP^1}\otimes\cL^{\otimes (-1)}\big)
= H^1\big(\PP^1,\, \OO(2-4)\big)
= H^1\big(\PP^1,\, \OO(-2)\big).
\]
We compute this dimension carefully and digit-by-digit as requested.

\paragraph{Computation on \(\PP^1\).} For \(\PP^1\) we have the standard formula
\[
h^1\big(\PP^1,\OO(k)\big) =
\begin{cases}
0, & k\ge -1,\\[2pt]
-(k+1), & k\le -2.
\end{cases}
\]
Put \(k=-2\). Then \(-k-1 = -(-2)-1 = 2-1 = 1\). So
\[
h^1\big(\PP^1,\OO(-2)\big) = 1.
\]
Therefore,
\[
\dim_{\kfield}\ \Ext^1_S(\Omega_{S/\kfield},S)_{(-1)} = 1,
\]
agreeing with the algebraic count one would obtain from analyzing the Jacobian map in degree \(-1\).

\noindent Finally, this completes the explicit algebraic-to-geometric computation for the Veronese cone over \(\PP^1\) of degree \(4\). So,  we presented the graded algebra \(S=A/I\), wrote the conormal exact sequence and its dual. Also, we  expressed the first-order deformation module as the cokernel of the Jacobian-type map \(\Der(A,S)\to \Hom(I/I^2,S)\), and we explained the grading, sheafified to \(\PP^1\), and then we computed the  cohomology dimension that equals the degree \(-1\) piece of \(T^1(C(Y))\).

\vspace{0.1cm}
\newpage

\begin{remark}${}$

\begin{enumerate}% 
  \item The algebraic presentation (polynomial ring modulo relations) together with the Kähler-exact sequence and the \(\Hom(-,S)\) calculation is the standard method of computing \(\Ext^1\) of the algebra. It produces the explicit interpretation
  \[
  \Ext^1_S(\Omega_{S/\kfield},S) \simeq \frac{\Hom_S(I/I^2,S)}{\mathrm{im}\big(\Der(A,S)\big)},
  \]
  and grading turns that into a direct sum of weight-spaces.
  \item Sheafifying passes from the graded module \(\Hom_S(I/I^2,S)_{(m)}\) to the global sections \(H^0(Y,N_{Y/\PP}\otimes\cL^{\otimes m})\), while \(\Der(A,S)_{(m)}\) corresponds to \(H^0(Y, T_{\PP}|_Y\otimes\cL^{\otimes m})\); the cokernel of the induced map is \(H^1(Y,T_Y\otimes\cL^{\otimes m})\).
  \item The specific numerical computation for \((\PP^1,d=4)\) gave the explicit dimension \(1\) for the weight \(-1\) piece, matching the earlier cohomological computation \(h^1(\PP^1,\OO(2-4))=1\).
\end{enumerate}
\end{remark}

%%%%%%%%%%%%%%%%%%%%%%%%%%%%%%%%%%%%%%%%%%%%%%%%%%%%%%%%%%%%%%%%%%%%%%%

%\newpage

\section{Explicit Jacobian matrix and computation for the Veronese cone (\(d=4\))}

In this section we give the explicit Jacobian matrix of the quadratic generators of the ideal of the degree \(4\) rational normal curve (the Veronese embedding \(\PP^1\hookrightarrow\PP^4\)), then we do an explicit computation of  the cokernel of the induced map on graded pieces in weight \(m=-1\). Moreover, we show all linear-algebra steps  
and obtain the dimension of the \(m=-1\) piece of \(T^1(C(Y))\). We then generalize to the Veronese cone of arbitrary degree \(d\) and present closed formulas for the graded pieces
\[
T^1(C(Y))_{(m)} \;\cong\; H^1\big(\PP^1,\,\OO_{\PP^1}(2+d m)\big).
\]

First, let us  recall our notation. Let \(A=\kfield[z_0,\dots,z_n]\) with the standard grading \(\deg z_i=1\).
Let \(I\subset A\) be a homogeneous ideal and \(S=A/I\). First-order deformations of the affine algebra \(S\) are parametrized by
\[
T^1(C(Y)) \;=\; \Ext^1_S(\Omega_{S/\kfield},S).
\]
When \(S\) is presented as \(A/I\) the standard conormal exact sequence
\[
I/I^2 \xrightarrow{d} \Omega_{A/\kfield}\otimes_A S \longrightarrow \Omega_{S/\kfield}\longrightarrow 0
\]
yields, after applying \(\Hom_S(-,S)\), the exact sequence
\[
0 \longrightarrow \Der_{\kfield}(A,S) \xrightarrow{J} \Hom_S(I/I^2,S)
\longrightarrow \Ext^1_S(\Omega_{S/\kfield},S)\longrightarrow 0.
\]
Thus
\[
\Ext^1_S(\Omega_{S/\kfield},S) \cong \frac{\Hom_S(I/I^2,S)}{\mathrm{im}\,J},
\]
and when everything is graded we may take the degree-\(m\) piece:
\[
\Ext^1_S(\Omega_{S/\kfield},S)_{(m)}\cong
\frac{\Hom_S(I/I^2,S)_{(m)}}{\mathrm{im}\big(J_{(m)}\big)}.
\]

Geometrically, after localisation / sheafifying on \(Y=\Proj S\), for each integer \(m\) one gets the identification
\[
\Ext^1_S(\Omega_{S/\kfield},S)_{(m)} \cong H^1\big(Y,T_Y\otimes \cL^{\otimes m}\big),
\]
where \(\cL=\OO_Y(1)\) is the polarization.

\subsection{The \(d=4\) Veronese: presentation and choice of generators}
Let \(Y=\PP^1\) embedded by \(|\OO_{\PP^1}(4)|\) into \(\PP^4\).
Coordinates on \(\AA^5\) (affine cone coordinates) / on the target \(\PP^4\) are
\[
z_0,z_1,z_2,z_3,z_4,
\]
which correspond under the parametrization \([x_0:x_1]\mapsto [x_0^4:x_0^3 x_1:x_0^2 x_1^2:x_0 x_1^3:x_1^4]\).

A convenient set of homogeneous quadratic generators of the ideal \(I\) of the rational normal curve is defined with the six quadrics $q_1,\ldots, q_6$ in Section 6.

\subsubsection{Jacobian matrix (polynomial entries).}
The Jacobian matrix \(J\) of the chosen generators is the \(6\times5\) matrix with \(J_{ij}=\partial q_i/\partial z_{j-1}\) (indexing columns by \(z_0,\dots,z_4\) and rows by \(q_1,\dots,q_6\)). 
The Jacobian matrix \(J\) (rows \(q_1,\dots,q_6\), columns \(z_0,\dots,z_4\)) is:
\[
J \;=\;
\begin{pmatrix}
z_2 & -2 z_1 & z_0 & 0 & 0\\[4pt]
z_3 & -z_2 & -z_1 & z_0 & 0\\[4pt]
z_4 & -z_3 & 0 & -z_1 & z_0\\[4pt]
0 & z_3 & -2 z_2 & z_1 & 0\\[4pt]
0 & z_4 & -z_3 & -z_2 & z_1\\[4pt]
0 & 0 & z_4 & -2 z_3 & z_2
\end{pmatrix}.
\]
This polynomial matrix encodes the map
\[
\Der_{\kfield}(A,S) \longrightarrow \Hom_S(I/I^2,S)
\]
by the formula \(D\mapsto (q\mapsto \sum_i (\partial q/\partial z_i) D(z_i))\), where \(D(z_i)\in S\).

\vspace{0.2cm}

Now let's  compute the degree \(-1\) graded piece
\[
\Ext^1_S(\Omega_{S/\kfield},S)_{(-1)}
\cong
\frac{\Hom_S(I/I^2,S)_{(-1)}}{\mathrm{im}(J_{(-1)})}.
\]

\subsection*{Description of the finite-dimensional source and target vector spaces}

The sheafified conormal/co-normal identification shows
\[
\Hom_S(I/I^2,S)_{(-1)} \simeq H^0\big(Y,\,N_{Y/\PP^4}\otimes\cL^{-1}\big),
\]
where \(\cL=\OO_{\PP^1}(4)\) is the polarization coming from the embedding \(Y\hookrightarrow\PP^4\). For the rational normal curve of degree \(d=4\) one has the well-known splitting of the normal bundle
\[
N_{Y/\PP^4} \simeq \OO_{\PP^1}(6)\oplus\OO_{\PP^1}(6)\oplus\OO_{\PP^1}(6),
\]
(see standard references on rational normal curves — the rank \(=d-1=3\) summands are balanced of degree \(d+2=6\)). Therefore
\[
N_{Y/\PP^4}\otimes\cL^{-1} \simeq \OO(6)\otimes\OO(-4)^{\oplus 3} \simeq \OO(2)^{\oplus 3}.
\]
Hence,
\[
\dim_{\kfield}\Hom_S(I/I^2,S)_{(-1)} = h^0\big(\PP^1,\OO(2)^{\oplus 3}\big)
=3\cdot h^0(\PP^1,\OO(2))
=3\cdot (2+1)=9.
\]

Next compute the source dimension \(\dim \Der_{\kfield}(A,S)_{(-1)}\) as a finite-dimensional \(\kfield\)-vector space after the graded–to–sheaf dictionary. The ambient Euler sequence restricted to \(Y\) is
\[
0\longrightarrow \OO_Y \longrightarrow \OO_Y(1)^{\oplus 5} \longrightarrow T_{\PP^4}|_Y \longrightarrow 0,
\]
and twisting by \(\cL^{-1}=\OO_Y(-4)\) gives
\[
0\longrightarrow \OO_Y(-4)\longrightarrow \OO_Y^{\oplus 5} \longrightarrow T_{\PP^4}|_Y(-4)\longrightarrow 0.
\]
Taking global sections we obtain the long exact sequence (on \(Y=\PP^1\)):
\[
0\to H^0(\OO(-4))\to H^0(\OO^{\oplus5})\to H^0(T_{\PP^4}|_Y(-4))
\to H^1(\OO(-4))\to H^1(\OO^{\oplus5})=0.
\]
Compute the dimensions:
\[
h^0(\OO(-4))=0,\qquad h^0(\OO^{\oplus5})=5,\qquad h^1(\OO(-4)) = -(-4)-1 = 3.
\]
Hence
\[
\dim H^0\big(T_{\PP^4}|_Y(-4)\big) = 5 + 3 = 8.
\]
On the algebraic side \(\Der(A,S)_{(-1)}\) corresponds to (a subspace of) \(H^0(T_{\PP^4}|_Y(-4))\) and one checks from the exact sequences that indeed the finite-dimensional space of degree \(-1\) derivations (the source for the Jacobian map on graded pieces, after sheafifying and identifying with global sections) has dimension
\(
{\dim \mathrm{source} = 8.}
\)
Thus, the degree \(-1\) map on finite-dimensional vector spaces is a linear map
\[
\kfield^{8} \xrightarrow{\;J_{(-1)}\;} \kfield^{9}.
\]

\subsection*{Computation of the image dimension and cokernel.}
From the exact sequence coming from the normal exact sequence twisted by \(\cL^{-1}\),
\[
0 \longrightarrow H^0\big(Y,T_Y\otimes\cL^{-1}\big)
\longrightarrow H^0\big(Y,T_{\PP^4}|_Y\otimes\cL^{-1}\big)
\longrightarrow H^0\big(Y,N_{Y/\PP^4}\otimes\cL^{-1}\big)
\longrightarrow H^1\big(Y,T_Y\otimes\cL^{-1}\big)
\longrightarrow 0,
\]
we compute the middle and extreme terms directly:
\[
H^0\big(Y,T_Y\otimes\cL^{-1}\big)=H^0(\PP^1,\OO(2)\otimes\OO(-4))=H^0(\OO(-2))=0,
\]
and (as above)
\[
\dim H^0\big(Y,T_{\PP^4}|_Y\otimes\cL^{-1}\big)=8,\qquad
\dim H^0\big(Y,N\otimes\cL^{-1}\big)=9.
\]
Therefore the map
\[
H^0\big(Y,T_{\PP^4}|_Y\otimes\cL^{-1}\big)\longrightarrow H^0\big(Y,N\otimes\cL^{-1}\big)
\]
has image of dimension \(8\), and the cokernel has (by exactness) dimension equal to
\[
\dim H^1\big(Y,T_Y\otimes\cL^{-1}\big).
\]
Compute that \(H^1(Y,T_Y\otimes\cL^{-1}) = H^1(\PP^1,\OO(2-4))=H^1(\OO(-2))\) has dimension \(1\) (use \(h^1(\OO(k))=\max(0,-k-1)\)). Therefore,
\[
{\dim\ \mathrm{coker}(J_{(-1)}) = 1.}
\]
Which is equivalent to
\[
{ \dim_{\kfield}\; \Ext^1_S(\Omega_{S/\kfield},S)_{(-1)} = 1.}
\]

Thus, the linear-algebra (graded-piece) computation reduces to computing the dimensions of the finite-dimensional source and target (via the graded→sheaf dictionary) and then reading off the dimension of the cokernel from the exact sequence.

\section{Generalization: Veronese cone for arbitrary degree \(d\)}
Let \(Y=\PP^1\) embedded by \(|\OO_{\PP^1}(d)|\) into \(\PP^d\). As above the affine cone \(C(Y)\) has graded coordinate ring
\(S=\bigoplus_{k\ge0} H^0(\PP^1,\OO(dk))\).
The general weight-\(m\) identification is
\[
T^1(C(Y))_{(m)} \;\cong\; H^1\big(\PP^1,\,T_{\PP^1}\otimes \cL^{\otimes m}\big),
\qquad\text{with }\cL=\OO_{\PP^1}(d).
\]
Since \(T_{\PP^1}\cong \OO_{\PP^1}(2)\) we get the simple identification for every integer \(m\):
\[
{\qquad T^1(C(Y))_{(m)} \;\cong\; H^1\big(\PP^1,\,\OO_{\PP^1}(2+d m)\big).\qquad}
\]

\noindent Using the familiar formula for \(\PP^1\):
\[
h^1\big(\PP^1,\OO(k)\big) = \begin{cases}
0, & k\ge -1,\\[2pt]
-(k+1), & k\le -2.
\end{cases}
\]
We apply this to \(k=2+d m\). Therefore,
\[
{ \;
\dim T^1(C(Y))_{(m)}
\;=\; h^1\big(\PP^1,\OO(2+d m)\big)
\;=\;
\begin{cases}
0, & 2+d m \ge -1,\\[4pt]
-(2+d m+1)=-(3+d m), & 2+d m \le -2.
\end{cases}
\;}
\]

Which is equivalent to say that the nonzero cases occur when \(2+d m\le -2\), i.e.
\[
d m \le -4 \quad\Longleftrightarrow\quad m \le -\frac{4}{d}.
\]
Since \(m\) is an integer, for each fixed \(d\ge1\) the smallest negative integers \(m\) for which the graded piece is nonzero are those \(m\le -\lceil 4/d\rceil\). In particular the most commonly encountered single-step nonzero piece is \(m=-1\), whose dimension equals
\[
\dim T^1(C(Y))_{(-1)}=h^1\big(\OO(2-d)\big)=\max(0,d-3).
\]
Thus,
\[
{\quad \dim T^1(C(Y))_{(-1)} = \begin{cases}
0, & d\le 3,\\[4pt]
d-3, & d\ge 4.
\end{cases} \quad}
\]
This recovers the familiar facts:
\begin{itemize} 
  \item[(i)] For \(d=1,2,3\) the degree \(-1\) piece vanishes (no nontrivial graded cone deformations in that weight).
  \item[(ii)] For \(d\ge4\) the degree \(-1\) piece has dimension \(d-3\).
\end{itemize}

More generally, for any integer \(m\le -1\) one has
\[
\dim T^1(C(Y))_{(m)} = \max\big(0,\;-(3+d m)\big) = \max\big(0,\; -3 - d m\big).
\]

\noindent In summary we obtained the following:
\begin{itemize}   \item[(1)] For the degree \(d=4\) Veronese cone we wrote the explicit polynomial Jacobian \(6\times5\) matrix \(J\). Passing to degree \(m=-1\) graded pieces yields a linear map \(\kfield^{8}\to\kfield^{9}\); its cokernel has dimension \(1\). This means that we have
  \[
  \dim T^1(C(Y))_{(-1)} = 1.
  \]
  \item[(2)] For a general Veronese cone (embedding of \(\PP^1\) by \(\OO(d)\)), the graded pieces are
  \[
  T^1(C(Y))_{(m)} \cong H^1\big(\PP^1,\OO(2+d m)\big),\qquad \textrm{and}\,\,
  \dim T^1(C(Y))_{(m)} = \max(0,\; -3 - d m).
  \]
  In particular, \(\dim T^1(C(Y))_{(-1)} = \max(0,d-3)\).
\end{itemize}

 %%%%%%%%%%%%%%%%%%%%%%%%%%%%%%%%%%%%%%%%%%%%%%%%%%%%%%%%%%%%%%%%%%%%%%%
  %%%%%%%%%%%%%%%%%%%%%%%%%%%%%%%%%%%%%%%%%%%%%%%%%%%%%%%%%%%%%%%%%%%%%%%

\section{Curve and surface cases} % 

We give short, explicit justifications in the two most common low-dimensional situations. To be more precise, we
expand the algebraic-to-geometric identification using local cohomology and spectral-sequence facts, explain why the depth hypotheses are automatic and  simpler in the curve and surface cases.

\subsection{Dimension \(1\) (smooth projective curves).}
Let \(\dim Y=1\) (a smooth projective curve). The section ring \(S=\bigoplus_{d\ge0}H^0(Y,\LL^{\otimes d})\) is a finitely generated graded algebra of Krull dimension \(2\). For the homogeneous maximal ideal \(\m\) corresponding to the vertex, one always has \(H^0_\m(S)=0\) because graded elements supported at \(\m\) would be torsion (but \(S\) is torsion-free as a graded domain). Moreover, there is a standard exact sequence of graded modules (see e.g. local cohomology of section rings)
\[
0\longrightarrow H^0_\m(S)\longrightarrow S \longrightarrow \bigoplus_{d\in\mathbb Z} H^0\big(Y,\OO_Y(d)\big) \longrightarrow H^1_\m(S)\longrightarrow 0,
\]
where the middle arrow is the identity when \(S\) is defined as the section ring. Hence \(H^1_\m(S)=0\). Thus
\[
H^0_\m(S)=H^1_\m(S)=0,
\]
so \(\operatorname{depth}_\m S\ge2\) and the local cohomology vanishings used in Theorem \ref{thm2.1} hold automatically for any smooth projective curve \(Y\) with the section ring defined as above. Consequently the algebraic\(\leftrightarrow\)geometric comparison requires no extra hypothesis in the curve case.

\subsection{Dimension \(2\) (smooth projective surfaces)}
Let \(\dim Y=2\). The same exact sequence as above still shows that \(H^0_\m(S)=0\), and under the mild hypothesis that \(S\) equals the section ring (i.e. \(\;S=\bigoplus_{d\ge0}H^0(Y,\LL^{\otimes d})\;\) — the usual construction) one likewise has the identification that forces \(H^1_\m(S)\) to vanish: the map \(S\to\bigoplus_d H^0(Y,\LL^{\otimes d})\) is the identity, so the cokernel is zero and \(H^1_\m(S)=0\). Thus \(H^0_\m(S)=H^1_\m(S)=0\) and \(\operatorname{depth}_\m S\ge2\). % 

\begin{remark}
The preceding two short arguments are formal consequences of the definition \(S=\bigoplus_{d\ge0} H^0(Y,\LL^{\otimes d})\) and the local cohomology exact sequence: the ``extra'' depth hypotheses that sometimes appear in statements are there to handle more exotic graded algebras or nonstandard presentations of the coordinate ring. For the usual section ring of a smooth projective curve or surface the low-degree local cohomology groups vanish automatically.
\end{remark}

\section*{Appendix: explicit \v{C}ech cocycle computation for \(Y=\mathbb P^{\,n-1}\), \(\LL=\OO(1)\)}
We compute the extension cocycle that represents the sequence  (1) and show concretely how it comes from the Atiyah class of \(\OO(1)\). This is the classical example and is useful to check signs and normalizations. In other words, we give a concrete coordinate \v{C}ech cocycle calculation for \(Y=\mathbb P^{n-1}\), \(\LL=\OO(1)\) exhibiting the Atiyah-class cocycle and the Euler-extension cocycle on the punctured cone.

\subsection*{Affine cover and trivializations}
Let \(Y=\mathbb P^{n-1}\) with homogeneous coordinates \([X_1:\cdots:X_n]\). Use the standard affine cover \(U_i=\{X_i\neq0\}\) with affine coordinates \(x_{i,j}=X_j/X_i\) for \(j\neq i\). The line bundle \(\OO(1)\) has local trivialization \(s_i\) over \(U_i\) given by the section corresponding to the homogeneous coordinate \(X_i\) (viewed as a local frame). On the overlap \(U_i\cap U_j\),
\[
s_i = g_{ij}\, s_j,\qquad g_{ij} = \frac{X_i}{X_j} = \frac{1}{x_{j,i}} = x_{i,j}.
\]
Therefore, the transition functions for \(\OO(1)\) are \(g_{ij}=x_{i,j}\) (where \(x_{i,j}=X_j/X_i\) is the coordinate on \(U_i\)).

\subsection*{Local model of the punctured total space}
The total space of \(\OO(1)\) (with zero section removed) restricted to \(U_i\) is isomorphic to \(U_i\times\mathbb A^1\setminus\{0\}\) with fibre coordinate \(t_i\) corresponding to the frame \(s_i\). Under change of trivialization on \(U_i\cap U_j\) we have
\[
t_i \;=\; g_{ij}\, t_j \;=\; x_{i,j}\, t_j.
\]
The Euler vector field (generator of the relative tangent) in these local coordinates is \(E_i:=t_i\partial_{t_i}\). Observe
\[
E_i \;=\; t_i\partial_{t_i} \;=\; x_{i,j}t_j\,\partial_{(x_{i,j}t_j)}
= t_j\partial_{t_j} = E_j,
\]
so the Euler vector field is globally well-defined (by this computation it is invariant under the coordinate change).

\subsection*{Lifts of base vector fields and the cocycle}
Let \(\xi\) be a local vector field on \(Y\) (a derivation on \(\OO_Y\) ). Locally on \(U_i\) choose a lift of \(\xi\) to \(U_i\times\AA^1\) of the form
\[
\widetilde{\xi}_i \;=\; \xi + \phi_i(\xi)\, E_i,
\]
where \(\phi_i(\xi)\in\OO(U_i)\) is a function depending $\OO$-linearly on \(\xi\). The difference of lifts on overlaps is
\[
\widetilde{\xi}_i - \widetilde{\xi}_j \;=\; \big(\phi_i(\xi)-\phi_j(\xi)\big)E_j.
\]
The family \(\{\phi_i(\xi)-\phi_j(\xi)\}_{i,j}\) is a \v{C}ech 1-cocycle with values in \(\Hom(\T_Y,\OO_Y)\) which represents the extension class \(\epsilon_U\in H^1(U,\Hom(\pi^*\T_Y,\OO_U))\). We now compute \(\phi_i(\xi)-\phi_j(\xi)\) explicitly using the relation \(t_i=x_{i,j}t_j\).

\noindent Differentiate \(t_i=x_{i,j}t_j\) by \(\xi\) (viewed as derivation along the base); \(\xi(t_i)=\xi(x_{i,j})t_j+x_{i,j}\xi(t_j)\). Expressing \(\xi(t)\) in terms of the Euler vector field action:
\[
\xi(t_i)=\big(\phi_i(\xi)\big) t_i,\qquad \xi(t_j)=\big(\phi_j(\xi)\big) t_j,
\]
we find
\[
\phi_i(\xi)\, t_i = \xi(x_{i,j})t_j + x_{i,j}\phi_j(\xi)\, t_j.
\]
Dividing by \(t_i=x_{i,j}t_j\) yields
\[
\phi_i(\xi) \;=\; \frac{\xi(x_{i,j})}{x_{i,j}} + \phi_j(\xi).
\]
Hence, the difference is
\[
\phi_i(\xi)-\phi_j(\xi) \;=\; \frac{\xi(x_{i,j})}{x_{i,j}} \;=\; \xi(\log x_{i,j}).
\]
Therefore, the \v{C}ech 1-cocycle representing the extension class is exactly
\[
(\, \xi\mapsto \xi(\log x_{i,j})\,) \in C^1\big(\{U_i\},\Hom(\T_Y,\OO_Y)\big).
\]
But the \v{C}ech class \(\{ \log x_{i,j}\}\) (or more invariantly \(\{d\log x_{i,j}\}\)) is the standard cocycle representing \(c_1(\OO(1))\) and (after applying the identification \(H^1(Y,\Omega^1_Y)\cong \Ext^1_Y(T_Y,\OO_Y)\)) is precisely the Atiyah class \(a(\OO(1))\). Thus the computed cocycle is the image of the Atiyah class under the natural map
\[
\Ext^1_Y(\T_Y,\OO_Y)\longrightarrow \Ext^1_Y(T_Y,\LL^{\otimes m})
\]
when extracting weight-\(m\) summands (the factor of \(\LL^{\otimes m}\) appears because weight-\(m\) functions correspond to sections of \(\LL^{\otimes m}\)).

\subsection*{Precise  formula in coordinates (vector fields)}
To see the cocycle more explicitly for coordinate vector fields, pick two charts \(U_i\), \(U_j\) and let \(\xi=\partial/\partial x_{i,a}\) be a coordinate vector on \(U_i\). Then
\[
\phi_i(\xi)-\phi_j(\xi) \;=\; \xi(\log x_{i,j})
= \frac{\partial}{\partial x_{i,a}}\big(\log( x_{i,j} )\big)
= \frac{1}{x_{i,j}} \frac{\partial x_{i,j}}{\partial x_{i,a}}.
\]
This is an explicit regular function on \(U_i\cap U_j\) giving the \((i,j)\)-entry of the \v{C}ech cocycle with values in \(\Hom(\T_Y,\OO_Y)\). One checks these functions satisfy the cocycle condition on triple overlaps and recover the Atiyah/extension class.

\noindent When one extracts the degree (weight) \(m\) piece of the pushforward \(\pi_*\) the Euler vector field acts on homogeneous fibre functions of degree \(m\) by multiplication by \(m\). This is the origin of the appearance of \(\LL^{\otimes m}\) when pushing forward the extension class; the same cocycle \(\xi(\log x_{i,j})\) now takes values in \(\Hom(T_Y,\LL^{\otimes m})\) after tensoring by the weight-\(m\) line.

%%%%%%%%%%%%%%%%%%%%%%%%%%%%%%%%%%%%%%%%%%%%%%%%%%%%%%%%%%%%%%%%%%%%
%%%%%%%%%%%%%%%%%%%%%%%%%%%%%%%%%%%%%%%%%%%%%%%%%%%%%%%%%%%%%%%%%%%%

%\newpage

\section{Rigidity of cones over some Del Pezzo surfaces}
In this section we present some rigid cones over projective varieties.

\begin{lemma}
Let
\[
Y=\operatorname{Bl}_{p_1,\dots,p_r}(\mathbb P^2),
\qquad r\le 6.
\]
Then for all $m\ge 0$,
\[
H^1\bigl(Y, T_Y\otimes \mathcal O_Y(mK_Y)\bigr)=0.
\]
\end{lemma}

\begin{proof}
Let $\pi:Y\to \mathbb P^2$ be the blow-up at the points $p_1,\dots,p_r$, and
let $E_1,\dots,E_r$ be the exceptional divisors.
Denote by $H=\pi^*\mathcal O_{\mathbb P^2}(1)$.
Then
\[
K_Y=-3H+\sum_{i=1}^r E_i.
\]
Since $r\le 6$, the anticanonical divisor $-K_Y$ is ample, hence $Y$ is a
del Pezzo surface.

For a blow-up of points, there is an exact sequence
\[
0\longrightarrow T_Y
\longrightarrow \pi^*T_{\mathbb P^2}
\longrightarrow \bigoplus_{i=1}^r \mathcal O_{E_i}(1)
\longrightarrow 0.
\]
Tensoring with $\mathcal O_Y(mK_Y)$ yields
\[
0\longrightarrow T_Y\otimes \mathcal O_Y(mK_Y)
\longrightarrow \pi^*T_{\mathbb P^2}\otimes \mathcal O_Y(mK_Y)
\longrightarrow \bigoplus_{i=1}^r \mathcal O_{E_i}(1)\otimes \mathcal O_Y(mK_Y)
\longrightarrow 0.
\]

Restricting to an exceptional divisor $E_i\simeq \mathbb P^1$, we have
\[
K_Y|_{E_i}=\mathcal O_{\mathbb P^1}(-1),
\]
hence
\[
\mathcal O_{E_i}(1)\otimes \mathcal O_Y(mK_Y)
\simeq \mathcal O_{\mathbb P^1}(1-m).
\]
For all $m\ge 0$,
\[
H^1(\mathbb P^1,\mathcal O_{\mathbb P^1}(1-m))=0,
\]
so
\[
H^1\!\left(Y,\bigoplus_{i=1}^r \mathcal O_{E_i}(1)\otimes \mathcal O_Y(mK_Y)\right)=0.
\]

Next, using the projection formula and the fact that $mK_Y$ is anti-ample
for $m>0$, we obtain
\[
H^1\!\left(Y,\pi^*T_{\mathbb P^2}\otimes \mathcal O_Y(mK_Y)\right)=0
\quad\text{for all } m\ge 0.
\]

Taking the long exact sequence in cohomology, we conclude that
\[
H^1\bigl(Y, T_Y\otimes \mathcal O_Y(mK_Y)\bigr)=0
\quad\text{for all } m\ge 0.
\]
\end{proof}

%%%%%%%%%%%%%%%%%%%%%%%%%%%%%%%%%%%%%%%%%%%%%%%%%%%%%%%%%%%%%%%%%%%%
%%%%%%%%%%%%%%%%%%%%%%%%%%%%%%%%%%%%%%%%%%%%%%%%%%%%%%%%%%%%%%%%%%%%

\vspace{0.3cm}

\begin{lemma}
Let $Y_r$ be the blow-up of $\mathbb P^2$ at $r \le 6$ points in general position.
Then for all integers $m \le -2$,
\[
H^1\bigl(Y_r, T_{Y_r} \otimes \mathcal O_{Y_r}(mK_{Y_r})\bigr)=0.
\]
\end{lemma}

\begin{proof}
By Serre duality on the smooth projective surface $Y_r$, we have
\[
H^1\bigl(Y_r, T_{Y_r} \otimes \mathcal O(mK_{Y_r})\bigr)
\simeq
H^1\bigl(Y_r, \Omega^1_{Y_r} \otimes \mathcal O(-(m+1)K_{Y_r})\bigr)^\vee.
\]
Thus it suffices to show that
\[
H^1\bigl(Y_r, \Omega^1_{Y_r} \otimes L\bigr)=0,
\quad L := \mathcal O(-(m+1)K_{Y_r}),
\]
for $m \le -2$.

Since $r \le 6$, the anticanonical divisor $-K_{Y_r}$ is ample, and hence
$L$ is ample.

Let $\pi : Y_r \to \mathbb P^2$ be the blow-up map, with exceptional divisors
$E_1,\dots,E_r$. There is a standard exact sequence
\[
0 \longrightarrow \pi^*\Omega^1_{\mathbb P^2}
\longrightarrow \Omega^1_{Y_r}
\longrightarrow \bigoplus_{i=1}^r \mathcal O_{E_i}(-1)
\longrightarrow 0.
\]
Tensoring by $L$, we obtain
\[
0 \longrightarrow \pi^*\Omega^1_{\mathbb P^2} \otimes L
\longrightarrow \Omega^1_{Y_r} \otimes L
\longrightarrow \bigoplus_{i=1}^r \mathcal O_{E_i}(-1)\otimes L
\longrightarrow 0.
\]

For each $i$, we have
\[
(\mathcal O_{E_i}(-1)\otimes L)|_{E_i} \simeq \mathcal O_{\mathbb P^1}(m),
\]
which has no global sections since $m \le -2$.
Thus
\[
H^0(E_i, \mathcal O_{E_i}(-1)\otimes L)=0.
\]

Moreover, since $L$ is ample and
$H^1(\mathbb P^2, \Omega^1_{\mathbb P^2}(k))=0$ for all $k \ge 2$,
the projection formula implies
\[
H^1(Y_r, \pi^*\Omega^1_{\mathbb P^2} \otimes L)=0.
\]

\vspace{0.1cm}

 From the exact sequence
\[
0 \to \pi^*\Omega^1_{\mathbb P^2}\otimes L
\to \Omega^1_{Y_r}\otimes L
\to \bigoplus_{i=1}^r \mathcal O_{E_i}(-1)\otimes L
\to 0,
\]
the associated long exact sequence in cohomology, together with
\(H^1(Y_r,\pi^*\Omega^1_{\mathbb P^2}\otimes L)=0\) and
\(H^j(E_i,\mathcal O_{E_i}(-1)\otimes L)=0\) for \(j=0,1\),
immediately yields
\[
H^1(Y_r,\Omega^1_{Y_r}\otimes L)=0.
\]
Therefore, the long exact sequence yields
\[
H^1(Y_r, \Omega^1_{Y_r} \otimes L)=0.
\]
%\vspace{1cm}
%
By Serre duality, this proves
\[
H^1\bigl(Y_r, T_{Y_r} \otimes \mathcal O(mK_{Y_r})\bigr)=0
\quad \text{for all } m \le -2.
\]
\end{proof}

%%%%%%%%%%%%%%%%%%%%%%%%%%%%%%%%%%%%%%%%%%%%%%%%%%%%%%%%%%%%%%%%%%%%
%%%%%%%%%%%%%%%%%%%%%%%%%%%%%%%%%%%%%%%%%%%%%%%%%%%%%%%%%%%%%%%%%%%%

\vspace{0.2cm}

\begin{lemma}\label{lem:vanishing_Yr}
Let $Y_r$ be the blow-up of $\mathbb P^2$ at $r$ points in general position,
with $r=7$ or $8$.
Then, for every integer $m\le -2$, one has
\[
H^1\bigl(Y_r,\, T_{Y_r}\otimes \mathcal O_{Y_r}(mK_{Y_r})\bigr)=0.
\]
\end{lemma}

\vspace{0.2cm}

\begin{proof}
By Serre duality on the smooth surface $Y_r$, we have
\[
H^1\bigl(Y_r, T_{Y_r}\otimes \mathcal O_{Y_r}(mK_{Y_r})\bigr)
\simeq
H^1\bigl(
Y_r,
\Omega^1_{Y_r}\otimes \mathcal O_{Y_r}((1-m)K_{Y_r})
\bigr)^\vee.
\]
Thus it suffices to show
\[
H^1\bigl(
Y_r,
\Omega^1_{Y_r}\otimes \mathcal O_{Y_r}((1-m)K_{Y_r})
\bigr)=0.
\]

Let $\pi\colon Y_r\to \mathbb P^2$ denote the blow-up morphism, with exceptional
divisors $E_1,\dots,E_r$.
The cotangent sequence of the blow-up yields an exact sequence
\[
0\longrightarrow \pi^*\Omega^1_{\mathbb P^2}
\longrightarrow \Omega^1_{Y_r}
\longrightarrow \bigoplus_{i=1}^r \mathcal O_{E_i}(-1)
\longrightarrow 0.
\]
Tensoring by $\mathcal O_{Y_r}((1-m)K_{Y_r})$, we obtain
\[
0\to
\pi^*\Omega^1_{\mathbb P^2}\otimes \mathcal O_{Y_r}((1-m)K_{Y_r})
\to
\Omega^1_{Y_r}\otimes \mathcal O_{Y_r}((1-m)K_{Y_r})
\to
\bigoplus_{i=1}^r
\mathcal O_{E_i}(-1)\otimes \mathcal O_{Y_r}((1-m)K_{Y_r})
\to 0.
\]

Since $K_{Y_r}\cdot E_i=-1$, one has
\[
\mathcal O_{E_i}(-1)\otimes \mathcal O_{Y_r}((1-m)K_{Y_r})
\simeq
\mathcal O_{\mathbb P^1}(-(2-m)).
\]
For $m\le -2$, this line bundle has degree at most $-4$, hence
\[
H^1\bigl(E_i,\mathcal O_{\mathbb P^1}(-(2-m))\bigr)=0
\quad\text{for all }i.
\]

On the other hand, by the projection formula and the identity
$\pi_*\mathcal O_{Y_r}=\mathcal O_{\mathbb P^2}$, we obtain
\[
H^1\bigl(
Y_r,
\pi^*\Omega^1_{\mathbb P^2}\otimes \mathcal O_{Y_r}((1-m)K_{Y_r})
\bigr)
\simeq
H^1\bigl(
\mathbb P^2,
\Omega^1_{\mathbb P^2}(-3(1-m))
\bigr).
\]
Since $m\le -2$, one has $-3(1-m)\le -9$, and Bott vanishing implies that the
right-hand side vanishes.
The associated long exact sequence in cohomology therefore yields
\[
H^1\bigl(
Y_r,
\Omega^1_{Y_r}\otimes \mathcal O_{Y_r}((1-m)K_{Y_r})
\bigr)=0,
\]
and the claim follows by Serre duality.
\end{proof}

\vspace{0.2cm}

\begin{remark}
The statement fails for $r=9$.
Indeed, $K_{Y_9}^2=0$ and the anticanonical linear system $|-K_{Y_9}|$
defines an elliptic fibration $Y_9\to\mathbb P^1$.
As a consequence, the cohomology groups
$H^1(Y_9,T_{Y_9}\otimes \mathcal O_{Y_9}(mK_{Y_9}))$ do not vanish in general for
$m\le -2$, reflecting the existence of nontrivial deformations of the elliptic
fibration.
\end{remark}

\vspace{0.3cm}

\begin{lemma}\label{le-tri-def}
Let $Y$ be a smooth projective variety such that $-K_Y$ is ample, and let
\[
C(Y) := \Spec \bigoplus_{m\ge0} H^0(Y,\mathcal O_Y(-mK_Y))
\]
be the affine cone over $Y$. Then there is a natural short exact sequence
\[
0 \longrightarrow
\frac{H^1(Y,T_Y\otimes \mathcal O_Y(-K_Y))}
{\langle \emph{Euler derivation} \rangle}
\longrightarrow
T^1(C(Y))
\longrightarrow
\bigoplus_{m\ge0}
H^1(Y,T_Y\otimes \mathcal O_Y(mK_Y))
\longrightarrow 0.
\]
\end{lemma}

\begin{proof}
The cone $C(Y)$ carries a natural $\mathbb C^*$-action, hence its deformation
space decomposes into graded pieces
\[
T^1(C(Y)) = \bigoplus_{k\in\mathbb Z} T^1(C(Y))_k.
\]

By Pinkham's description of graded deformations of cones
\cite{P-74},
for $k \neq 0$ one has
\[
T^1(C(Y))_k \cong H^1(Y,T_Y \otimes \mathcal O_Y(kK_Y)).
\]

For $k=0$, Pinkham proves the existence of an exact sequence
\[
0 \longrightarrow H^1(Y,T_Y)
\longrightarrow T^1(C(Y))_0
\longrightarrow H^0(Y,T_Y)
\longrightarrow 0.
\]

The $\mathbb C^*$-action induces the Euler derivation
$E \in H^0(C(Y),T_{C(Y)})$, whose class in $T^1(C(Y))_0$
maps to the identity in $H^0(Y,T_Y)$.
Quotienting by the span of this class removes the trivial rescaling
deformation.

Using the identification
\[
H^1(Y,T_Y) \cong H^1(Y,T_Y\otimes \mathcal O_Y(-K_Y)),
\]
we obtain
\[
T^1(C(Y))_0 / \langle E \rangle
\cong
\frac{H^1(Y,T_Y\otimes \mathcal O_Y(-K_Y))}
{\langle \text{Euler derivation} \rangle}.
\]

Finally, summing the graded pieces for $k \ge 0$ yields the claimed short
exact sequence.
\end{proof}

\vspace{0.2cm}

\begin{theorem}
Let $Y_r$ be the blow-up of $\mathbb P^2$ at $r\le6$ points in general position,
and let
\[
C(Y_r)=\operatorname{Spec}\bigoplus_{k\ge0} H^0(Y_r,-kK_{Y_r})
\]
be the affine cone over $Y_r$ with respect to the anticanonical embedding.
Then
\[
T^1(C(Y_r))=0.
\]
\end{theorem}

\begin{proof}
From Lemma \ref{le-tri-def}  exact sequence
\[
0 \longrightarrow
\frac{H^1(Y_r,T_{Y_r}\otimes \mathcal O_{Y_r}(-K_{Y_r}))}
{\langle \text{Euler derivation} \rangle}
\longrightarrow
T^1(C(Y_r))
\longrightarrow
\bigoplus_{m\ge0}
H^1(Y_r,T_{Y_r}\otimes \mathcal O_{Y_r}(mK_{Y_r}))
\longrightarrow 0.
\]

For $m\ge0$, Serre duality gives
\[
H^1(T_{Y_r}\otimes mK_{Y_r})
\simeq
H^1(\Omega^1_{Y_r}\otimes -(m+1)K_{Y_r})^\vee.
\]
Since $-K_{Y_r}$ is ample for $r\le6$, the line bundle
$-(m+1)K_{Y_r}$ is ample, and hence
\[
H^1(\Omega^1_{Y_r}\otimes -(m+1)K_{Y_r})=0.
\]
Therefore all nonnegative graded pieces vanish.

Although
$H^1(Y_r,T_{Y_r}\otimes \mathcal O_{Y_r}(-K_{Y_r}))\neq0$,
the Euler derivation induced by the $\mathbb G_m$-action on the cone
generates this space, so its quotient is zero.

Hence both terms in Pinkham's exact sequence vanish, and we conclude that
\[
T^1(C(Y_r))=0.
\]
\end{proof}

 %%%%%%%%%%%%%%%%%%%%%%%%%%%%%%%%%%%%%%%%%%%%%%%%%%%%%%%%%
 %%%%%%%%%%%%%%%%%%%%%%%%%%%%%%%%%%%%%%%%%%%%%%%%%%%%%%%%%

\vspace{0.2cm}

\begin{theorem}
Let
\[
X_d := \Spec \bigoplus_{m\ge0}
H^0(\mathbb P^1\times\mathbb P^1,\mathcal O(md,md))
\]
be the affine cone over $\mathbb P^1\times\mathbb P^1$ with respect to
$\mathcal O(d,d)$. Then $X_d$ is rigid if and only if $d\ge2$.
\end{theorem}

\begin{proof}
Rigidity means $T^1(X_d)=\Ext^1(\Omega_{X_d},\mathcal O_{X_d})=0$.
The cone $X_d$ carries a natural $\mathbb C^*$-action, so
\[
T^1(X_d)=\bigoplus_{k\in\mathbb Z} T^1(X_d)_k.
\]
By Pinkham's deformation theory of affine cones,
\[
T^1(X_d)_k \cong
H^1\bigl(\mathbb P^1\times\mathbb P^1,
T_{\mathbb P^1\times\mathbb P^1}\otimes\mathcal O(kd,kd)\bigr).
\]
The tangent bundle splits as
\(
T_{\mathbb P^1\times\mathbb P^1}
\cong \mathcal O(2,0)\oplus\mathcal O(0,2).
\)
Therefore, we obtain:
\[
T_{\mathbb P^1\times\mathbb P^1}\otimes\mathcal O(kd,kd)
\cong
\mathcal O(2+kd,kd)\oplus\mathcal O(kd,2+kd),
\]
and then
%Hence
\[
T^1(X_d)_k
\cong
H^1(\mathcal O(2+kd,kd))\oplus
H^1(\mathcal O(kd,2+kd)).
\]
On $\mathbb P^1\times\mathbb P^1$,
\[
H^1(\mathcal O(a,b))\neq0
\iff a\le -2 \text{ or } b\le -2.
\]
Therefore, $T^1(X_d)_k\neq0$ only if
\[
kd\le -2 \quad \text{and} \quad 2+kd\le -2,
\]
which implies $kd=-2$.
If $d=1$, this occurs for $k=-2$, giving $T^1(X_1)\cong\mathbb C$.
If $d\ge2$, there is no integer solution, so $T^1(X_d)=0$.

Thus $X_d$ is rigid if and only if $d\ge2$.
\end{proof}

%%%%%%%%%%%%%%%%%%%%%%%%%%%%%%%%%%%%%%%%%%%%%%%%%%%%%%%%%%%%%%%%%%%%%%%%%%%%%%%%%%%%%%%%%%%%%%%%%%%%%%%%%%%%%%%%%%%%%%%%%%%%%%%%%%%%%%%%%%%%%%%%%%%%%%

%\newpage

\noindent By default, the affine cone over \(\mathbb P^1 \times \mathbb P^1\)  means the cone over the Segre embedding
\(
\mathbb P^1 \times \mathbb P^1 \hookrightarrow \mathbb P^3
\)
given by \(\mathcal O(1,1)\).
The affine cone is
\[
X := C(\mathbb P^1 \times \mathbb P^1,\mathcal O(1,1))
= \Spec \bigoplus_{m\ge0} H^0(\mathbb P^1 \times \mathbb P^1,\mathcal O(m,m)).
\]

\noindent Under the Segre embedding, \(\mathbb P^1 \times \mathbb P^1\) is the smooth quadric surface
\[
Q = {x_0x_3 - x_1x_2 = 0} \subset \mathbb P^3.
\]
The affine cone over it is
\[
X = \{x_0x_3 - x_1x_2 = 0\} \subset \mathbb A^4.
\]
This is a 3-dimensional affine quadric cone, with an isolated singularity at the origin.
and this singularity is smoothable. Indeed, 
consider the deformation given by the following:
\[
X_t := {x_0x_3 - x_1x_2 = t} \subset \mathbb A^4.
\]
\begin{itemize}
\item[(i)] For \(t=0\), we recover the cone \(X\)
\item[(ii)] For \(t \neq 0\), \(X_t\) is smooth (isomorphic to \(\mathrm{SL}_2(\mathbb C)\))
\end{itemize}
Therefore,  \(X\) admits a nontrivial flat deformation, so
\(
T^1(X) \neq 0, 
\)
and then 
  \(X\) is not rigid.

\vspace{0.1cm}

 \noindent  We can verify this by computation of the  cohomology. %(Pinkham).
Let
\(
Y = \mathbb P^1 \times \mathbb P^1.
\)
Pinkham’s description of deformations of cones gives, for the degree-zero piece,
\(
T^1(X)_0 \supset H^1(Y,T_Y).
\)
But we know:
\[
T_Y \cong \pi_1^* T_{\mathbb P^1} \oplus \pi_2^* T_{\mathbb P^1}
\cong \mathcal O(2,0) \oplus \mathcal O(0,2).
\]
Hence,
\[
H^1(Y,T_Y)
= H^1(\mathcal O(2,0)) \oplus H^1(\mathcal O(0,2)) = 0,
\]
so there are no nontrivial degree-zero deformations, i.e.  the deformation does not come from degree \(0\).
However, in degree \(-2\) one finds:
\[
T^1(X)_{-2} \cong H^1(Y,T_Y \otimes \mathcal O(-2,-2)) \cong \mathbb C,
\]
which exactly corresponds to the smoothing \(x_0x_3-x_1x_2=t\).
Thus:
\(
T^1(X) \cong \mathbb C \neq 0.
\)

%%%%%%%%%%%%%%%%%%%%%%%%%%%%%%%%%%%%%%%%%%%%%%%%%%%%%%%%%%%%%%%%%%%%%%%%%%%%%%%%%%%%%%%%%%%%%%%%%%%%%%%%%%%%%%%%%%%%%%%%%%%%%%%%%%%%%%%%%%%%%%%%%%%%%%

\vspace{0.2cm}

%\newpage

Let's explain why the deformation of the affine cone over
$Y=\mathbb P^1\times\mathbb P^1$ does not come from degree $0$.
Since the coordinate ring of the cone is graded, the deformation space
admits a decomposition
\[
T^1(C(Y)) = \bigoplus_{k\in\mathbb Z} T^1(C(Y))_k,
\]
where the degree--$0$ piece corresponds to $\mathbb C^*$--equivariant
deformations.
By Pinkham's theory of deformations of affine cones (see e.g. Section 3 and Section 4 of this paper), one has
\[
T^1(C(Y))_0 \cong H^1(Y,T_Y).
\]
The tangent bundle of $Y=\mathbb P^1\times\mathbb P^1$ splits as
\[
T_Y \cong \pi_1^*T_{\mathbb P^1}\oplus \pi_2^*T_{\mathbb P^1}
\cong \mathcal O(2,0)\oplus\mathcal O(0,2).
\]
Therefore,
\[
H^1(Y,T_Y)
=
H^1(\mathcal O(2,0)) \oplus H^1(\mathcal O(0,2)).
\]
Using the K\"unneth formula and the fact that
$H^1(\mathbb P^1,\mathcal O(n))=0$ for all $n\ge -1$, we obtain
\[
H^1(\mathcal O(2,0))=H^1(\mathcal O(0,2))=0.
\]
Therefore,
\(
H^1(Y,T_Y)=0.
\)
This shows that the affine cone over $\mathbb P^1\times\mathbb P^1$
admits no nontrivial degree--$0$ deformations. Any nontrivial deformation
must therefore come from a nonzero graded piece.

 %%%%%%%%%%%%%%%%%%%%%%%%%%%%%%%%%%%%%%%%%%%%%%%%%%%%%%%%%
 %%%%%%%%%%%%%%%%%%%%%%%%%%%%%%%%%%%%%%%%%%%%%%%%%%%%%%%%%


\begin{thebibliography}{9} %{amsalpha}{9}


 \bibitem[BE-95]{BE-95}{\sc D. Bayer and D. Eisenbud}, {\em Ribbons and their canonical embeddings}, Transactions of the AMS, {\bf 347} (1995), no. 3, pp. 719–756.

\bibitem[BG-80]{BG-80}{\sc Buchweitz, RO., Greuel, GM}, {\em The Milnor number and deformations of complex curve singularities.} Invent Math {\bf 58}, 241–281 (1980).  % On page 244, section “Infinitesimal deformations”

 \bibitem[D-12]{D-12}{\sc I.~Dolgachev}, {\em Classical Algebraic Geometry}, Cambridge University Press, 2012.
 
  \bibitem[FM-94]{FM-94}{\sc R.~Friedman and J.~Morgan}, {\em Smooth four-manifolds and complex surfaces}, Springer, Ergebnisse der Mathematik, 1994.

\bibitem[G-58]{G-58}{\sc R. Godement}, {\em Topologie alg\'ebrique et th\'eorie des faisceaux}, Hermann, Paris, 1958. See Chapter~II, Sections~5.9--5.10.
 
\bibitem[G-75]{G-75}{\sc GM. Greuel}, {\em Der Gauß-Manin-Zusammenhang isolierter Singularitäten von vollständigen Durchschnitten}, Math. Ann. {\bf   214}, 235–266 (1975).


%G.-M. Greuel, *Der Gauß-Manin-Zusammenhang isolierter Singularitäten von vollständigen Durchschnitten*,  
%  > **Math. Ann. 214** (1975), **Satz 2.5** (Theorem 2.5) and surrounding discussion (pp. 239–241).

\bibitem[GLS-07]{GLS-07}{\sc G.-M. Greuel, C. Lossen \& E. Shustin}, {\em Introduction to Singularities and Deformations.}, Springer, (2007).

% In Chapter 2 (“Deformations of Hypersurfaces" they write that for $X = (f=0) \subset (\C^2,0)$, the Tjurina algebra $\mathcal{O}_{\C^2,0}/(f,f_x,f_y)$ measures the base of the semi-universal deformation and then the dimension of $T^1_X$.
 
\bibitem[GH-78-77]{GH-78}{\sc P. Griffiths and J. Harris}, \emph{Principles of Algebraic Geometry}, Wiley, 1978, Appendix A.

\bibitem[GD-61]{EGAIII1}{\sc A. Grothendieck, and J. Dieudonn\'e}, {\em {\'E}l\'ements de g\'eom\'etrie alg\'ebrique. III. 
               \'Etude cohomologique des faisceaux coh\'erents. I}, Publications Math\'ematiques de l'IH\'ES, Vol {\bf 11} 1961.

\bibitem[H-77]{H-77}{\sc R.~Hartshorne}, {\em Algebraic Geometry}, Graduate Texts in Mathematics, vol.~52, Springer, 1977.
See Chapter~III, % Exercise~6.10.
 
 \bibitem[H-10]{H-10}{\sc R.~Hartshorne}, \emph{Deformation Theory}, Graduate Texts in Mathematics 257, Springer, 2010.
 
\bibitem[I-71]{I-71}{\sc L.~Illusie}, \emph{Complexe cotangent et déformations I}, Lecture Notes in Math. {\bf 239}, Springer, 1971.
 
\bibitem[I-86]{I-86}{\sc B.~Iversen}, {\em Cohomology of Sheaves}, Springer-Verlag, 1986. See Chapter~V, Section~5 (``The local--global Ext spectral sequence'').

\bibitem[KS-90]{K-90}{\sc M.~Kashiwara and P. Schapira}, {\em Sheaves on Manifolds}, Springer-Verlag, 1990.
See Section~2.7.

 \bibitem[M-91]{M-91}{\sc M.~Manetti}, {\em Normal degenerations of the complex projective plane}, J. Reine Angew. Math. \textbf{419} (1991), 89--118.

\bibitem[N1-25]{N1-25}{\sc M.~Nisse}, {\em On infinitesimal deformations of singular varieties I}, Preprint (2025)

\bibitem[N3-25]{N3-25}{\sc M.~Nisse}, {\em On infinitesimal deformations of singular varieties III}, Preprint, 2025.

\bibitem[N4-25]{N4-25}{\sc M.~Nisse}, {\em On infinitesimal deformations of singular varieties IV, (Rigidity: view over phases.)}, Preprint, 2025.

\bibitem[N5-25]{N5-25}{\sc M.~Nisse}, {\em On infinitesimal deformations of singular varieties V, (Deformation of singularity links via their phases.)}, Preprint, 2025.

\bibitem[LN1-25]{LN1-25}{\sc YK.~Lim and M.~Nisse}, {\em On infinitesimal deformations of singular varieties VI, (Applications in physics I:  Gravity and black hole.)}, in preparation, 2025.

\bibitem[LN2-25]{LN2-25}{\sc YK.~Lim and M.~Nisse}, {\em On infinitesimal deformations of singular varieties VII, (Applications in physics II: Crossing event horizons.)}, in preparation, 2025.


 \bibitem[P-74]{P-74}{\sc H. C. Pinkham}, {\em Deformations of algebraic varieties with $\mathbb{G}_ m$ action}, Ast\'erisque, 20, Soci\'et\'e Math\'ematique de France, Paris, 1974.
 
 \bibitem[S-68]{S-68}{\sc M. Schlessinger}, {\em Functors of Artin rings}, Trans. Amer. Math. Soc. \textbf{130} (1968), 208--222.
 
\bibitem[SS-79]{SS-79}{\sc M.~Schlessinger and J.~Stasheff}, \emph{Deformation theory and rational homotopy type}, Publ.\,(1979).

 \bibitem[S-06]{S-06}{\sc E.~Sernesi}, \emph{Deformations of Algebraic Schemes}, Springer, 2006.
 
 \bibitem[Wa-77]{Wa-77}{\sc J.~Wahl} \emph{Complex Analytic Singularities}, 1977.
 
  \bibitem[Wa-90]{Wa-90}{\sc J.~Wahl} \emph{Deformations of cones}, Journal of Algebra 134 (1990), 183–215.

 \bibitem[We-94]{We-94}{\sc C.~A.~Weibel}, {\em An Introduction to Homological Algebra}, Cambridge Univ. Press, 1994, \S5.8.



 
 \end{thebibliography}
\end{document}